\documentclass[12pt]{amsart}
\usepackage[margin=1.2in]{geometry}
\usepackage{latexsym,amssymb,amsmath,curves} 
\usepackage[usenames,dvipsnames]{xcolor}
\usepackage{pdfpages}

\theoremstyle{plain}
 \newtheorem{theorem}{Theorem}[section]
 \newtheorem{proposition}[theorem]{Proposition}
  \newtheorem{lemma}[theorem]{Lemma}
  \newtheorem{corollary}[theorem]{Corollary}

 \newtheorem{definition}[theorem]{Definition}
  \newtheorem{example}[theorem]{Example}

  \newtheorem{remark}[theorem]{Remark}

\numberwithin{equation}{section}

\def\RR{{\mathbb R}}

\newcommand{\old}[1]{}

\begin{document}

\title[Shellability of face posets from electrical networks]{Shellability of face posets of electrical networks and  the CW poset property}

\author{Patricia Hersh}
\address{Department of Mathematics, University of Oregon, Eugene, OR, 97403}
\email{plhersh@uoregon.edu}

\author{Richard Kenyon}
\address{Department of Mathematics, Yale University, New Haven, CT, 06511}
\email{richard.kenyon@yale.edu}

\thanks{P. Hersh  is supported by NSF grants DMS-1500987 and DMS-1953931.  R. Kenyon is supported by NSF grant DMS-1713033 and the Simons Foundation award 327929.}

\begin{abstract}
We   prove  
a conjecture of Thomas Lam that the face posets of stratified spaces 
of planar resistor networks 
are shellable.  These posets are called uncrossing partial orders.  
This shellability result combines 
with Lam's previous result that these same posets are Eulerian to imply that they are CW posets, namely that they  are face posets of regular CW complexes.    Certain subposets of uncrossing partial orders  are shown to be isomorphic to type A Bruhat order intervals; our shelling is shown to coincide on these intervals with a Bruhat order shelling which was constructed by  Matthew Dyer using a reflection order.  

Our shelling for uncrossing posets also yields an explicit shelling for each interval in the face posets of the edge product spaces of phylogenetic trees, namely in the  Tuffley posets, by virtue of each interval in a Tuffley poset being isomorphic to an interval in an uncrossing poset.  This yields a  more explicit proof of the result of Gill, Linusson, Moulton and Steel  that the CW decomposition of Moulton and Steel for the edge product space of phylogenetic trees is a  regular CW decomposition.  
\end{abstract}

\maketitle

\section{Introduction}

We prove
a  conjecture of Thomas Lam from \cite{La14a}  that partially ordered sets
known as uncrossing posets have dual posets that 
are lexicographically  shellable.  This implies that the uncrossing posets themselves are also 
shellable.     This conjecture of Lam is proven  in Theorem ~\ref{EL-shell-proof}.
Specifically, we prove  that these  uncrossing posets 
are  dual EC-shellable (see Definition ~\ref{ec-def}). 
 Combining this  with 
Lam's result in \cite{La14a} that these posets are Eulerian (see Definition ~\ref{Eulerian-def}), 
we conclude that
these are CW posets (see Definition ~\ref{CW-def}), namely  that they 
are face posets of regular CW complexes.  Moreover, general
properties of lexicographic shellings allow us
also to conclude that each closed
interval in an  uncrossing poset  is also  a CW poset.  

These uncrossing  posets, denoted $P_n$ for $n\ge 2$,  naturally arise 
as face posets of  stratified spaces of 
planar electrical networks given by planar 
graphs  (as discussed for instance in \cite{Ke} and \cite{La14b})  for planar graphs 
that are ``well-connected'' (a notion defined for instance in  \cite{CIM}) with $n$ boundary
nodes.
Our result 
that these posets  are shellable is suggestive that these stratified spaces
may be well behaved topologically, and in particular may 
be regular CW complexes with each 
cell closure  homeomorphic to a closed ball; indeed this property of being a regular CW complex has just been announced in a new preprint 
of Galashin, Karp and Lam, namely  \cite{GKL19}, a development that came subsequent to our present
paper.     A proof that the closure of  the  big cell   is homeomorphic
to a closed ball was announced and outlined earlier  in \cite{GKL}.  Further related results
appear in \cite{GKL18}. 
Our shellability result  for uncrossing posets serves  as a combinatorial first step towards the question of understanding  the 
homeomorphism type of all cell closures for all planar electrical networks,  whether or not the
networks  are  well-connected, and our result   expands the combinatorial understanding of the cell structure.

Another consequence of our shelling for the uncrossing posets is a shelling for each interval in the face poset for the edge product space of phylogenetic trees, namely in  the Tuffley poset (see Definition ~\ref{Tuffley-def}).  
We give this  shelling for each interval of the Tuffley poset  in  Corollary 
~\ref{Tuffley-corollary}.   The main result in \cite{GLMS} is a  proof of the existence of a shelling for each interval in the Tuffley poset, but that paper left open the question of constructing such a shelling.   In the course of giving an explicit shelling  construction, we obtain an 
independent proof of shellability for each interval.    

The shelling existence result in \cite{GLMS}  is used within \cite{GLMS}   to prove that  the CW decomposition for  the  edge product space of phylogenetic trees given in \cite{MS}  is a regular CW decomposition.  Our shelling for the uncrossing poset yields an explicit construction of a shelling for each interval in the Tuffley poset, hence  also a more explicit  proof that the CW decomposition of \cite{MS} is a regular CW decomposition.  We do this by  showing that each interval in the Tuffley poset is an interval in an uncrossing poset and  then using the fact that any shelling of an entire poset that is induced  by a dual  EC-labeling (or by a dual EL-labeling) by definition  also  induces an explicit  shelling on each interval by restricting the labeling to the interval. 
Through  this connection to Tuffley posets, our work could also shed some new light on the edge product space of phylogenetic trees (in other words for an important
compactification of the tree space studied e.g. in \cite{BHV}).

Section ~\ref{BG-section} gives background  needed  both to understand our results and their motivations and  also needed  for their proofs.  Then Section ~\ref{Lam-section} gives a high-level view of  the proof of shellability for  uncrossing posets,  with the key technical lemmas (where a lot of the hard work is done) relegated to Section ~\ref{key-lemma-section}.  The application to Tuffley posets and edge product spaces of phylogenetic trees is given  in Section ~\ref{Tuffley-section}; this section includes quite a bit of detail in an effort to make this section relatively self-contained and reasonably accessible to readers from areas such as mathematical biology (without the need to read Section  ~\ref{key-lemma-section} in order  to understand Section ~\ref{Tuffley-section}).

The authors thank the referee for helpful feedback on an earlier version of the paper.

\subsection{Description of the uncrossing posets}

Denote  the uncrossing poset on $n$ wires by  $P_n$.  
Figure 1 shows the uncrossing poset  $P_3$.   

      \begin{figure}[h]\label{poset-P_3}
        \begin{picture}(325,482)(-80,-15)
 
       \put(77,441){\line(1,0){36}}
      \put(80,462){\line(1,-1){27}}
      \put(83,434){\line(1,1){28}}

        \put(83,434){\circle*{4}}
              \put(113,441){\circle*{4}}
       \put(77,441){\circle*{4}}
       \put(108,434){\circle*{4}}
       \put(80,462){\circle*{6}}
       \put(111,462){\circle*{4}}


 \put(70,462){$1$}
 \put(115,462){$2$}
 \put(118,439){$3$}
 \put(112,425){$1$}
 \put(72,425){$2$}
 \put(66,439){$3$}

\put(40,445){$\hat{1} = $}
       \put(95,450){\circle{100}}
       
       \curve(5,392,  80,420)
       \curve(95,392,  95,420)
       \curve(180,392,  110,420)

               \put(95,365){\circle{100}}             
      \put(80,377){\line(1,-1){27}}
\curve(77, 356, 94, 361, 111,377)
\curve(83,349, 98, 356, 113,356)

        \put(83,349){\circle*{4}}
              \put(113,356){\circle*{4}}
       \put(77,356){\circle*{4}}
       \put(108,349){\circle*{4}}
       \put(80,377){\circle*{6}}
       \put(111,377){\circle*{4}}

\put(70,379){$1$}

        \put(-5,365){\circle{100}}
\curve(-17, 349,  -15, 363, -20, 377)
\curve(8,349, 5, 363,  11, 377)
\curve(-23,356, 13, 356)

        \put(-17,349){\circle*{4}}
              \put(13,356){\circle*{4}}
       \put(-23,356){\circle*{4}}
       \put(8,349){\circle*{4}}
       \put(-20,377){\circle*{6}}
       \put(11,377){\circle*{4}}
 
 \put(-30,379){$1$}
 
        \put(195,365){\circle{100}}
      \put(183,349){\line(1,1){28}}
\curve(180,377, 197, 361,  213, 356)
\curve(177,356,  192, 357, 208, 349)

        \put(183,349){\circle*{4}}
              \put(213,356){\circle*{4}}
       \put(177,356){\circle*{4}}
       \put(208,349){\circle*{4}}
       \put(180,377){\circle*{6}}
       \put(211,377){\circle*{4}}

        \put(170,379){$1$}
        
        \curve(-15,335,  -45,250)
        \curve(-5,335,  5,250)
        \curve(5,335, 60,250)
        \curve(15,335, 115,250)

        \curve(80,335, -35, 250)
        \curve(90,335,  10, 250)
        \curve(100,335,  170, 250)
        \curve(110,335,  230, 250)
        
        \curve(175, 335, 70,250) 
        \curve(185, 335, 125, 250) 
        \curve(195, 335, 180, 250)
        \curve(205, 335, 240, 250)

\put(65,220){\circle{100}}
\curve(81,232, 78, 221, 83, 211)
\curve(50,232, 56, 218, 53,204)
\curve(47,211, 62, 213, 78,204)

        \put(53,204){\circle*{4}}
              \put(83,211){\circle*{4}}
       \put(47,211){\circle*{4}}
       \put(78,204){\circle*{4}}
       \put(50,232){\circle*{6}}
       \put(81,232){\circle*{4}}

\put(40,234){$1$}

\put(125,220){\circle{100}}
\curve(110,232, 127, 217, 143,211)
\curve(107,211, 114,211, 113,204)
\curve(141,232, 135, 218, 138,204)

        \put(113,204){\circle*{4}}
              \put(143,211){\circle*{4}}
       \put(107,211){\circle*{4}}
       \put(138,204){\circle*{4}}
       \put(110,232){\circle*{6}}
       \put(141,232){\circle*{4}}

\put(100,234){$1$}

\put(5,220){\circle{100}}

\curve(-13,211, 6, 216, 21, 232)
\curve(23,211, 17, 210, 18, 204)
\curve(-10, 232, -4, 218,  -7, 204)

        \put(-7,204){\circle*{4}}
              \put(23,211){\circle*{4}}
       \put(-13,211){\circle*{4}}
       \put(18,204){\circle*{4}}
       \put(-10,232){\circle*{6}}
       \put(21,232){\circle*{4}}

\put(-20,234){$1$}

\put(-55,220){\circle{100}}
\curve(-73,211, -67, 221, -70,232)
\curve(-67,204, -51, 214, -37,211)
\curve(-39,232, -45, 218, -42,204)

        \put(-67,204){\circle*{4}}
              \put(-37,211){\circle*{4}}
       \put(-73,211){\circle*{4}}
       \put(-42,204){\circle*{4}}
       \put(-70,232){\circle*{6}}
       \put(-39,232){\circle*{4}}

\put(-80,234){$1$}

\put(185,220){\circle{100}}
\curve(170,232, 186, 216, 203,211)
\curve(173,204, 185, 207, 198,204)
\curve(201,232, 184, 216, 167,211)

        \put(173,204){\circle*{4}}
              \put(203,211){\circle*{4}}
       \put(167,211){\circle*{4}}
       \put(198,204){\circle*{4}}
       \put(170,232){\circle*{6}}
       \put(201,232){\circle*{4}}

\put(160,234){$1$}

\put(245,220){\circle{100}}
 \curve(230,232, 245, 228, 261,232)
\curve(227,211, 242, 211,  258,204)
\curve(233,204,  248, 211, 263,211)

        \put(233,204){\circle*{4}}
              \put(263,211){\circle*{4}}
       \put(227,211){\circle*{4}}
       \put(258,204){\circle*{4}}
       \put(230,232){\circle*{6}}
       \put(261,232){\circle*{4}}

\put(220,234){$1$}

\curve(-55, 190,  -35, 105)
\curve(-45, 190,  25, 105)

\curve(0, 190, -25, 105)
\curve(10, 190, 90, 105)

\curve(60, 190, 35, 105)
\curve(70, 190, 145, 105)

\curve(120, 190, 100, 105)
\curve(130, 190, 155, 105)

\curve(180, 190, 45, 105)
\curve(190, 190, 210, 105)

\curve(240, 190, 115, 105)
\curve(250, 190, 220, 105)

\put(95,75){\circle{100}}
\curve(80,87, 95, 83, 111,87)
\curve(77,66, 83,66,  83,59)
\curve(108,59, 108,66, 113,66)

        \put(83,59){\circle*{4}}
              \put(113,66){\circle*{4}}
       \put(77,66){\circle*{4}}
       \put(108,59){\circle*{4}}
       \put(80,87){\circle*{6}}
       \put(111,87){\circle*{4}}

\put(70,89){$1$}

\put(155,75){\circle{100}}
\curve(140,87, 168,59)
\curve(171,87, 168, 76, 173,66)
\curve(143,59, 143,66, 137,66)

        \put(143,59){\circle*{4}}
              \put(173,66){\circle*{4}}
       \put(137,66){\circle*{4}}
       \put(168,59){\circle*{4}}
       \put(140,87){\circle*{6}}
       \put(171,87){\circle*{4}}

\put(130,89){$1$}

\put(35,75){\circle{100}}
\curve(20,87, 23, 76, 17,66)
\curve(23,59, 35, 63, 48,59)
\curve(53,66, 48, 76,  51,87)

        \put(23,59){\circle*{4}}
              \put(53,66){\circle*{4}}
       \put(17,66){\circle*{4}}
       \put(48,59){\circle*{4}}
       \put(20,87){\circle*{6}}
       \put(51,87){\circle*{4}}

\put(10,89){$1$}

\put(-25,75){\circle{100}}
\curve(-43,66,  -37, 76, -40,87)
\curve(-37,59, -9,87)
\curve(-7,66, -13,66, -12,59)

        \put(-37,59){\circle*{4}}
              \put(-7,66){\circle*{4}}
       \put(-43,66){\circle*{4}}
       \put(-12,59){\circle*{4}}
       \put(-40,87){\circle*{6}}
       \put(-9,87){\circle*{4}}

\put(-50,89){$1$}

\put(215,75){\circle{100}}
\curve(200,87, 215, 83,  231,87)
\curve(203,59, 215, 61, 228,59)
\curve(233,66, 215, 66, 197,66)

        \put(203,59){\circle*{4}}
              \put(233,66){\circle*{4}}
       \put(197,66){\circle*{4}}
       \put(228,59){\circle*{4}}
       \put(200,87){\circle*{6}}
       \put(231,87){\circle*{4}}

\put(190,89){$1$}

\curve(75, -5, -20, 47)
\curve(85, -5, 35, 47)
\curve(95,-5, 95, 47)
\curve(105,-5, 150, 47)
\curve(115, -5, 200, 47)

\put(92,-25){$\hat{0}$}
          
        \end{picture}
        \caption{The Hasse diagram for $P_3$}   
\end{figure}

Let us first describe the elements of $P_n$.
We place $2n$ nodes around the boundary of a disk, 
then connect these nodes  in pairs using $n$ wires to do so.  
In addition 
to $P_n$ including all such wire diagrams with $n$ wires, 
we adjoin an element $\hat{0}$.
The elements of $P_n \setminus \{ \hat{0} \} $ 
are in natural bijection with  those permutations of $2n$ letters which are fixed point free involutions.  To see this, begin  by labeling the wire endpoints 
(proceeding clockwise around the boundary of the disk from a chosen  basepoint) with the integers $1,2,\dots ,2n$ assigned in ascending order;  the  
fixed point free involution associated to a wire diagram $D\in P_n$ consists of exactly the product of 
 2-cycles $(i,j)$ where $i$ and $j$ are the endpoints of a wire in $D$. 

Now let us define the order relation.  
 There is a unique maximal element $\hat{1}$ in $P_n$ given by a wire diagram $D$ 
  in which all  $n$ strands cross 
 each other.   See  the leftmost diagram in 
 Figure 2 for the case with $n=3$.  This  naturally corresponds to
  the fixed point free involution which exchanges $i$ with $n+i$ for each $i\in [1,n]$.  
 We may proceed   from an element $v$ to an element $u$  with $u < v$ 
  by taking a pair of  wires  that cross and locally uncrossing  the pair of wires in either of the two 
  possible ways.  This gives a cover relation $u\prec v$ if and only if this
  downward step decreases by exactly one the 
   total number of pairs of wires that  cross each other in the drawing
   for $u$ that minimizes this
   crossing number.  In other words, we  have 
 $u\prec v$   if and only if the uncrossing of a pair of wires in $v$  to obtain $u$ 
 does not introduce any double crossings of pairs of wires in $u$.  
 See the rightmost diagram in Figure 2  for an element 
 obtained by uncrossing a pair of wires in the fully crossed diagram $\hat{1}$, indeed yielding a 
 cover relation for $n=3$.  Notice that uncrossing this same pair of wires in the other direction would
 introduce a double crossing.

 \begin{figure}[h]\label{examples}
        \begin{picture}(325,50)(-80,75)
 
        \put(200,100){\circle{100}}
        
        \put(188,84){\circle*{4}}
              \put(218,91){\circle*{4}}
       \put(182,91){\circle*{4}}
       \put(211,84){\circle*{4}}
       \put(188,116){\circle*{4}}
       \put(212,116){\circle*{4}}

       \put(181,91){\line(1,0){36}}
      \put(211,84){\line(0,1){32}}
      \put(188,84){\line(0,1){32}}

\put(30,97){which covers wire diagram  }

\put(-45,97){$\hat{1} = $}
       \put(0,100){\circle{100}}
       
       \put(-18,91){\line(1,0){36}}
      \put(-15,112){\line(1,-1){27}}
      \put(-12,84){\line(1,1){28}}

        \put(-12,84){\circle*{4}}
              \put(18,91){\circle*{4}}
       \put(-18,91){\circle*{4}}
       \put(13,84){\circle*{4}}
       \put(-15,112){\circle*{4}}
       \put(16,112){\circle*{4}}

        \end{picture}
        \caption{Two wire diagrams with $n=3$ wires}
  \end{figure}

 This process of proceeding down cover relations naturally terminates at many different minimal elements given by the various wire  diagrams with no crossings.  
     A unique minimal element $\hat{0}$ is artificially adjoined to the poset, rendering the wire 
     diagrams  without any crossings  as the atoms of the resulting poset $P_n$.  
     See Lemma 
 ~\ref{noncross-down} for   a precise  combinatorial description for the cover relations 
 in $P_n$. 
     
     \begin{remark}
     The number of atoms in $P_n$ is the $n$-th Catalan number, namely is $\frac{1}{2n+1}{2n+1
     \choose n}$.
     \end{remark}
     
We observe in Corollary ~\ref{size-of-uncrossing-poset} that the number of elements in $P_n$  is 
$$1 + \frac{(2n)!}{n!2^n}.$$

\begin{remark}\label{rank-def}
{\rm
The uncrossing poset $P_n$ is graded by letting the rank of any  $D\ne \hat{0} $
be one more than the number of pairs of wires that cross each other in $D$, 
with the rank of $\hat{0}$ being 0.
}
\end{remark}

Our first main result, conjectured  in \cite{La14a} and proven as Theorem ~\ref{EL-shell-proof}
and Corollary ~\ref{CW-proof}, is as follows.

\begin{theorem}\label{shellab-and-cw}
The uncrossing poset $P_n$ is  EC-shellable for each $n\ge 2$.  Moreover, it is a CW poset.
\end{theorem}

The proof  exploits  a close relationship between $P_n$
and Bruhat order.  In particular, the proof  uses 
 Dyer's notion of 
reflection order from \cite{Dy93}  to guide the choice of edge labeling.    A large class of 
intervals, namely those preserving what  we call the start set of a wire diagram,  
are proven to be dual isomorphic to type A Bruhat intervals.  Our labeling coincides 
on these  intervals  with a  known 
Bruhat order reflection order  EL-labeling.  It might be tempting to think every 
interval should be isomorphic to a type A Bruhat order interval, but whether this is true  
seems to be a rather subtle 
question that remains open.

Another conceptual aspect  of the  proof
is  the establishment of an analogue  to 
the notion of the inversion pairs of a permutation (see Definition ~\ref{noncross-def} and Lemma ~\ref{noncross-down}), what 
we call the  noncrossing pairs of a wire diagram.  
%
%
This allows us to construct  and justify the validity of 
a shelling based on a variant of a  type A reflection order.

 Our second main result, appearing as Corollary ~\ref{Tuffley-corollary}, is our explicit shelling for each interval in the Tuffley poset.

\section{Background}\label{BG-section}

We review background on partially ordered sets, shellability, CW complexes and CW posets,  reflection orders,  Dyer's EL-shelling for Bruhat order, 
and finally the affine symmetric group.  This is 
done in preparation for the proof  of our main result in the following section.

\subsection{Partially ordered sets}

Denote by $u\prec v$ a {\bf cover relation} in a  partially ordered set (poset)  $P$, namely an order relation $u<v$ for a  pair of elements of $u,v\in P$ 
such that there does not exist $z\in P$ such that $u< z < v$.   We then say $v$ {\bf covers} $u$.  The 
{\bf Hasse diagram} of a poset $P$ is the  graph whose vertices are the elements of $P$ and whose edges are the cover relations $u\prec v$, typically drawn in the plane with each such edge proceeding upward from $u$ to $v$.  

If a poset has a unique minimal element,  denote this as $\hat{0}$.  Likewise, if a poset has a unique maximal element,  denote it as $\hat{1}$.  
An {\bf atom} is an element that covers 
$\hat{0}$ while a {\bf coatom} is an element covered by $\hat{1}$.  A {\bf chain} is a series
$u_1 < u_2 < \cdots < u_k$ of comparable poset elements.  A {\bf saturated chain from $u$ to $v$}  is a chain $u\prec u_1 \prec \cdots \prec u_k \prec v$ comprised of cover relations.  

A {\bf closed interval} $[u,v]$ is the subposet $\{ z\in P ~|~ u\le z \le v\}$.  An {\bf open interval} 
$(u,v)$ is the subposet $\{ z\in P ~|~ u < z < v \} $.
Any poset $P$ has a {\bf dual poset}, denoted $P^*$, with the same elements as $P$ and with $u\le v$ in $P^*$ if and only if $v\le u$ in $P$.

A poset is {\bf graded} if $u<v$ implies all maximal chains from $u$ to $v$ have the same number of cover relations, called the {\bf rank} of $[u,v]$.    If a graded poset has $\hat{0}$, then the {\bf rank} of each element $v$ is defined to be one more than the rank of each element $u$  covered by $v$, letting $\hat{0}$ have rank 0. 

The {\bf order complex} of a poset $P$ is the abstract simplicial complex, denoted $\Delta (P)$,  whose $i$-dimensional faces are the chains $u_0 < u_1 < \cdots < u_i $ of $i+1$ comparable poset elements.
Denote by $\Delta_P(u,v)$ the order complex of the open interval $(u,v)$ in $P$.  Notice that the saturated chains from $u$ to $v$ will be in natural bijection with the {\bf facets} (namely the maximal faces) of $\Delta_P(u,v)$, a fact that will be important to upcoming ``lexicographic shellings''.

The {\bf M\"obius function} $\mu_P$ of a poset $P$ is defined recursively by $\mu_P(u,u) = 1$ for each $u\in P$ and 
$$\mu_P(u,v) = -\sum_{u\le z < v} \mu_P(u,z)$$ for $u\ne v$.  
The  M\"obius function $\mu_P(u,v)$  is well-known to equal  the reduced Euler characteristic $\tilde{\chi }(\Delta_P(u,v)) $ (see \cite{Ro}).  In particular, if $\Delta_P(u,v) $ is homeomorphic to a $d$-sphere, this implies $\mu_P(u,v) = (-1)^d$.

\begin{definition}\label{Eulerian-def}
A graded poset is {\bf Eulerian } if  each $u<v$ has  $\mu_P(u,v) = (-1)^{r(u,v)}$ where $r(u,v)$
is the rank of $[u,v]$.
  A graded poset is {\bf thin} if each closed interval $[u,v]$ of rank 
2 has exactly 4 elements.  
\end{definition}

\begin{remark}\label{eulerian-then-thin}
If a graded poset is Eulerian, in particular it is thin.
\end{remark}

\subsection{Shellability}

Call the maximal faces of a simplicial complex the {\bf facets} of it.  
Define the link of a face $F$ in a simplicial complex $\Delta $, denoted
$lk_{\Delta } F$ to be the subcomplex 
 ${\rm lk}_\Delta F = \{ G \in \Delta | G\cap F = \emptyset \hspace{.1in} {\rm and } \hspace{.1in} F\cup G \in \Delta \} $.
 Define the {\bf combinatorial  closure} of a face $F$ in an abstract  simplicial complex $\Delta $, 
 denoted $\overline{F} $, to be the set of faces $G\in \Delta $ such that $G\subseteq F$.
 
A simplicial complex is {\bf pure of dimension $d$} if each facet is $d$-dimensional.  
A simplicial complex is {\bf shellable} if there is a total order $F_1,\dots ,F_k$ on its facets, called a 
{\bf shelling}, 
 such that for each $j\ge 2$ the subcomplex $\overline{F_j}\cap (\cup_{i<j} \overline{F_i} )$ is pure of dimension one less than the dimension of $F_j$.    
 
 Shellability of $\Delta $  is well known to imply homotopy equivalence to a wedge of spheres in a manner that is convenient for counting the spheres of each dimension, hence calculating reduced
 Euler characteristic.  
 A shelling for $\Delta $ also induces a  shelling  for the link of each face $F$ in 
 $\Delta $.   Since each open interval $(u,v)$ of a finite poset arises as the link of a face in its order
 complex,  shellability  
  is a useful tool for determining poset M\"obius function via its interpretation as 
    reduced Euler characteristic.

Now we turn to poset edge labelings $\lambda $, namely labelings of the cover relations 
$u\prec v $ with labels $\lambda (u,v)$ from an ordered label set.

\begin{definition}
Refer to  $u\prec v\prec w$ with $\lambda (u,v) \le \lambda (v,w)$ a  {\bf weak ascent}.  Call 
$u\prec v' \prec w$ with $\lambda (u,v') > \lambda (v',w)$ a {\bf descent}.
These terms are used 
both for the  saturated chains from $u$ to $w$  themselves and for the associated ordered pairs of labels. 
\end{definition}

\begin{definition}
A saturated chain $x \prec x_1\prec x_2\prec \cdots \prec x_k \prec y$ with 
$\lambda (x,x_1)\le \lambda (x_1,x_2) \le \cdots \le \lambda (x_k,y)$ is called a {\bf weakly 
ascending chain from $x$ to $y$}.  If instead we have 
$\lambda (x,x_1) > \lambda (x_1,x_2) > \cdots > \lambda (x_k,y)$, then this is called a 
{\bf descending chain from $x$ to $y$}.
\end{definition}

First we review the notion of EL-labeling, needed for our usage of Dyer's shelling of Bruhat 
order as an input to our shellability  proof for uncrossing orders.  Then we turn to the EC-labelings that 
will be our main tool for uncrossing orders.

\begin{definition}\label{EL-shell-def}
A labeling $\lambda $ on the cover relations of a poset $P$ with a totally ordered set
$\Lambda $  is an {\bf EL-labeling} if for each 
$u < v$ the following conditions are both met.
\begin{enumerate}
\item
There is a unique saturated chain $u\prec u_1\prec u_2\prec\cdots\prec u_k\prec v$ with weakly ascending
label sequence, namely with $\lambda (u,u_1)\le \lambda (u_1,u_2)\le \cdots \le \lambda (u_k,v)$. 
That is, 
there is a unique weakly ascending chain from $u$ to $v$ for each $u<v$.
\item
This label sequence is lexicographically smaller than the label sequence for every other saturated chain from $u$ to $v$.
\end{enumerate}
\end{definition}

For the uncrossing orders, we  will use a relaxation called EC-shelling  of the more well known 
notion of
 EL-shelling.  This idea of EC-labeling and EC-shellability  was developed  in \cite{Ko}
 (see also \cite{He03} for the convenient phrasing with topological ascents/descents we will use).   The key will be first  to relax the notions of ascent and descent in a way that still captures the same topological properties as the ascents and descents of an EL-labeling   while allowing a much wider array of possible labelings.

\begin{definition}\label{ec-def}
Given an edge labeling $\lambda $ of the cover relations in a graded poset, we say $u\prec v\prec w$ is a {\bf topological ascent} if the ordered pair $(\lambda (u,v), \lambda (v,w))$ of labels  is lexicographically smaller than all of the other label sequences 
for other saturated chains $u\prec v'  \prec w$  from $u$ to $w$.  
As a word of caution, notice that it might not be the case that $\lambda (u,v)\le \lambda (v,w)$. 
We say that $u\prec v \prec w$  is a {\bf topological descent} 
otherwise.   

An edge labeling is an {\bf $EC$-labeling}  
if each $u<w$ has a unique saturated chain from $u$ to $w$ comprised entirely of topological ascents. 
Note that this chain is in particular
the lexicographically smallest saturated chain from $u$ to $v$.
A poset with such a labeling is said to be {\bf $EC$-shellable}.
\end{definition}

The  saturated chains may be ordered lexicographically, and for 
the same reasons that EL-labelings induce shellings, the facet orderings
 induced by EC-labelings  will be shelling orders.  The topological descents will function in the shelling analogously to how descents function in an EL-shelling, and the topological ascents will function in the shelling just as ascents do in an EL-shelling: the topological descents $u\prec v \prec w$ in a saturated chain will index the vertices $v$ which may be omitted from the facet  corresponding to the saturated chain  to obtain  the  codimension one faces in the closure of the facet  that are shared with (closures of) earlier facets.  
 
 \begin{remark}
   Since  $\Delta (P) = \Delta (P^*)$, 
      it  suffices to construct an EL-labeling (or EC-labeling)   
   for $P^*$ 
   to deduce shellability for $\Delta (P)$. 
\end{remark}

\subsection{Face posets of regular CW complexes}

An {\bf open $m$-cell} is a  topological space  homeomorphic to the interior of 
an  $m$-dimensional ball $B^m$. Denote the {\bf closure} of a cell $\alpha $
by $\overline{\alpha }$.  

Given a topological space $X$  and a collection of disjoint open cells whose union is $X$, 
a {\bf characteristic map}  for the open $m$-dimensional cell $e_{\alpha }$ in $X$ 
is a continuous function 
$f_{\alpha } :B^m \rightarrow X$  
mapping 
 the interior of $B^m$ homeomorphically onto 
$e_{\alpha }$  and mapping the boundary of $B^m$ into a finite union of 
open cells, each of dimension less than $m$.

\begin{definition}\label{CW-def}  
{\rm 
A {\bf CW complex}  is a space $X$ and a collection of disjoint open cells $e_{\alpha }$
whose union is $X$ such that:
\begin{enumerate}
\item
$X$ is Hausdorff.
\item
For each open $m$-cell $e_{\alpha }$ of the collection, there exists a  characteristic map 
$f_{\alpha } :B^m \rightarrow X$ for $e_{\alpha }$. 
\item
A set $A$ is closed in $X$ if $A \cap \overline{e}_{\alpha }$ is closed in $\overline{e}_{\alpha }$
for each $\alpha $.
\end{enumerate} }
\end{definition}

A {\bf CW decomposition} of a topological space $X$ is a decomposition of $X$ into a disjoint union of open cells 
admitting  the structure of a CW complex in a way that coincides  with the given decomposition  into open cells and with 
the closure  operation on $X$ given by our choice of topology on $X$.  

\begin{definition}{\rm 
A  CW complex $K$ is a  {\bf regular CW complex} 
if  for each $m$ 
 there exists a  characteristic map $ f_{\alpha } : B^m \rightarrow  K$  
for each  open $m$-cell 
$e_\alpha $ in $K$ with the further property that 
  $f_{\alpha }$ restricts to a homeomorphism
from the boundary of $B^m$ onto a finite union of  open cells of strictly lower dimension than $m$. 
}
\end{definition}

For $K$ a regular CW complex, let $sd(K)$ denote the first barycentric subdivision of $K$, 
using the fact that
each cell closure in a regular CW complex is homeomorphic to a round ball to make sense of 
the notion of barycenter in this  level of generality and thereby to define $sd(K)$.
Notice for $K$ regular that $\Delta (F(K)\setminus \{ \hat{0} \} ) = sd(K) \cong K$.

\begin{definition}
The {\bf face poset} or {\bf closure poset}  of a CW complex $K$  is 
the  partial order $\le $ on the cells of  $K$ 
with $u\le v$ if and only if $u $ is contained in the closure of $v$.  This poset is denoted 
$F(K)$. 
\end{definition}

See \cite{Bj84} for the introduction of the next notion and the next  theorem.

\begin{definition}\label{CW-def}
A finite, graded poset $P$ is called a 
{\bf CW poset} if 
\begin{enumerate}
\item
$\hat{0} \in P$
\item
$P\setminus \hat{0} \ne \emptyset $
\item
$\Delta_P (\hat{0},u) $ is homeomorphic to a sphere of dimension $rk(u)-2$ for each $u\ne \hat{0}$
\end{enumerate}
\end{definition}

\begin{theorem}[Bj\"orner, \cite{Bj84}]
A finite poset
$P$ is a CW poset if and only if there exists a regular CW complex having $P$ as its face poset with 
$\hat{0}\in P$ representing the empty cell.
\end{theorem}

This combines with results from \cite{DK} to yield the following result, which is explained in  
Proposition 2.2 in \cite{Bj84}.  We have slightly rephrased this result  below  by using 
the fact 
that a shelling for a poset induces a shelling for each interval $[x,y]$ in it.  

\begin{theorem} \label{Danaraj-Klee} 
Any finite graded poset $P$ that is thin and shellable and has unique minimal element 
$\hat{0} $ as well as at least one additional element will be a CW poset.
\end{theorem}

\subsection{Reflection order EL-labeling for Bruhat order}
 This section reviews an EL-labeling of Dyer for Bruhat order.   We will need this in a special case  as an important 
 ingredient to our upcoming EC-shelling for uncrossing orders.  
 For further background on Coxeter groups and root systems, 
 see \cite{BB} and  \cite{Hu}.
 
 \begin{definition}
The  {\bf Bruhat order} is a partial order on the elements of a Coxeter group $W$
with cover relations $u\prec v$ when $v$ is obtained from $u$ by left multiplication by 
 a reflection that increases  ``length'' exactly by one. 
 
   In the case of the symmetric group, the {\bf reflections} are the transpositions $(i,j)$ and the 
 {\bf length} of any $\pi \in S_n$ is the number of inversions, that is, the 
 cardinality of $\{ 1\le i < j \le n | \pi (i) > \pi (j) \} $.

 \end{definition}

Given a Coxeter system $(W,S)$ with simple reflections $S$, let $T$ be the set of all its reflections 
$wsw^{-1}$ for $w\in W$ and $s\in S$.  The  reflections of $(W,S)$ are in natural bijection with the positive real roots.  
By way of this bijection,  any total order on positive roots  will also induce  a total order on reflections.   

Recall from Definition 2.1 in \cite{Dy93} and remarks  shortly thereafter:

\begin{definition}\label{reflection-order-definition}
{\rm A total order  $<$ on the positive roots of a
root system is called a {\bf reflection order} if each triple of roots $\alpha , \beta , c\alpha + d\beta $ for $c,d$ positive real numbers  satisfies $\alpha < c \alpha + d\beta < \beta $ or  $\beta < c\alpha + d \beta < \alpha $.}
\end{definition}

Dyer observes in \cite{Dy93}  that the following procedure will always yield reflection orders.

\begin{definition}\label{standard-reflection-orders}
{\rm 
For $W$ a (not necessarily finite)  reflection group, any total order 
on its  simple reflections gives rise to a {\bf lexicographic reflection order} 
$<_R$  on all positive roots  as follows.   
Each positive  root 
may be  written in a unique way  
as a positive sum of simple roots, hence as a vector in the coordinates  given by the  
simple roots.  We use the given order on simple roots  
to  order the coordinates in these vectors. 
Scale each resulting vector so that  its coordinates sum to 1.
  To obtain $<_R$, order these scaled vectors lexicographically.
}
\end{definition}
 
\begin{theorem}[Dyer, Proposition 4.3 in \cite{Dy93}]\label{Dyer-shelling}
Any reflection order  induces an EL-labeling on Bruhat order by labeling each
cover relation $u\prec v$ with the reflection $vu^{-1}$.
\end{theorem}

Next we describe a specific such  EL-labeling based on a particular reflection order.  This  will provide one of the important ingredients to our shelling for uncrossing posets.

\begin{corollary}
For the symmetric group $S_n$, the edge labeling $\lambda (u,v) = vu^{-1}$ induces an 
EL-labeling with respect to the following ordering on the set of labels, namely on the 
transpositions $(i,j) $ for $i<j$ in $S_n$:
$$(1,2) < (1,3) < \cdots < (1,n) < (2,3) < \cdots < (2,n) < \cdots < (n-1,n).$$
That is, for $i<j$ and $i'< j'$ we have $(i,j) < (i',j')$ if and only if we have either $i< i'$ or we have
$i=i'$ with $j<j'$.
\end{corollary}

We will also need the following characterization of cover relations in Bruhat order for $S_n$:

\begin{theorem}\label{Bruhat-cover-char}
There is a cover relation $\pi \prec (i,k)\cdot \pi $ for $i<k$  and for $\pi \in S_n$ 
in Bruhat order for $S_n$ 
if and only if the following conditions are both met:
\begin{enumerate}
\item
$\pi (i ) < \pi (k) $
\item
For each $j$ satisfying $i<j<k$ either $\pi (j) < \pi (i) $ or $\pi (k) < \pi (j)$.
\end{enumerate}
\end{theorem}

\begin{proof}
This is a special case of Proposition 4.6 in \cite{DH}, but we also include an elementary proof
an effort to keep our work self-contained.  This will require showing that $(i,k)\cdot \pi $ has exactly
one more inversion pair than $\pi $ does.  Notice that $(i,k)$ will be an inversion pair for 
$(i,k) \cdot \pi $ but not for $\pi $, since $\pi (i ) < \pi (k)$ whereas applying $(i,k)$ to $\pi $ to obtain 
$\tau = (i,k) \cdot \pi $ directly 
ensures we have $\tau (i) > \tau (k)$.  Also observe for each $j$ satisfying 
$i < j< k$ that $j$ forms an inversion pair with exactly one of the two letters $i,k$ in $\pi $, and it 
also forms an inversion pair with exactly one of the two letters $i,k$ in $\tau  $; specifically, the effect 
of applying $(i,k)$ to $\pi  $  is to exchange for each such $j$ whether it  
will be in an inversion pair with $i$ or with $k$.  
For every 
other pair $(i',k')$ with $i' < k'$, namely for each pair not equalling $(i,k)$ and not having $i'$ or $k'$ strictly intermediate
in value to $i$ and $k$, 
 notice that $(i',k')$  is an inversion pair for $\pi $ if and only if $(i',k')$  is an inversion pair 
for $\tau $. 
\end{proof}

\section{Proof of  Lam's  Shellability Conjecture}
\label{Lam-section}

Let us begin by  establishing  notational conventions that  will help us later 
to  assign names to wires in a wire diagram $D$ 
in an intrinsic way.  This  will be useful for giving the  poset Hasse diagram 
an edge labeling based on 
the names of the wires being uncrossed, a labeling that we 
will eventually  prove is an $EC$-labeling.

\begin{figure}[h]\label{diagram-with-labels}
        \begin{picture}(150,50)(-80,75)
 
        \put(0,100){\circle{100}}
       
       \put(-19,91){\line(1,0){36}}
      \put(-15,112){\line(1,-1){27}}
      \put(-12,84){\line(1,1){28}}

 \put(-32,113){$1_p$}
 
 \put(-15,112){\circle*{5}}
 \put(24,112){$2_p$}
 \put(25,88){$3_p$}
 
 \put(18,73){$1_v$}
 \put(-27,73){$2_v$}
 \put(-35,88){$3_v$}
          
        \end{picture}
        \caption{Wire endpoint labels for wire diagram $\hat{1} \in P_3$}     \end{figure}

Let us  choose a fixed wire endpoint in the fully crossed diagram with $n$ wires  to serve as a basepoint  which we label as $1_p$ and hold fixed as we proceed up a saturated chain progressively uncrossing more and more pairs of wires.  This fixed basepoint $1_p$
is depicted with a large dot to signify it is the basepoint.
This same fixed choice of 
basepoint $1_p$ is  used in all  the wire diagrams arising as poset elements.
Read clockwise around $D$ from starting position $1_p$.
Each time we encounter a new wire, label its endpoint we first encounter as $j_p$ where $j-1$ is the number of distinct wires previously encountered, as depicted in 
 Figure ~\ref{diagram-with-labels}.   
 When we reach the second endpoint of a wire whose first endpoint is $j_p$, label this second endpoint as  $j_v$.

\begin{definition}
  {\rm 
  Refer to the wire with   endpoints  labeled 
   $j_p$ and $j_v$  as wire $j$ or as the $j$-th wire.  

}
\end{definition}

\begin{definition}\label{word-def}
{\rm 
Define the {\bf word} of a wire diagram $D$ with $n$ strands, denoted $w(D)$, to be the word in the 
alphabet $1,2,\dots ,n$ obtained by starting at $1_p$ and reading clockwise the series of wires encountered.  Sometimes it is convenient to 
suppress the subscripts $p$ and $v$ on the letters,
 since these subscripts are redundant.   
}
\end{definition}

\begin{example}
The fully crossed diagram $D$ with 3 wires  has $w(D) = 123123$.  
The 5 fully uncrossed diagrams on 3 wires have words 
$$112233; 122331; 123321; 122133; 112332.$$  
\end{example}

\begin{remark}\label{two-types-uncrossings}
{\rm 
Recall from the definition of the uncrossing order $P_n$ that 
 a downward cover relation $D\rightarrow D'$  will  uncross a pair of wires $i,j$ such that doing so does not introduce any double-crossings.   There are two different potential ways to do this.
 Notice that for one of them,  $w(D')$ is obtained from $w(D)$ by 
 swapping the label $i_v$  with the label $j_v$ for $i<j$.  
 Uncrossing the wires in the other way will first swap the label $i_v$ with the label $j_p$ and then (if necessary) permute the names of the wires so that 
 the labels with subscript $p$ are encountered in increasing order as we proceed from left to right through $w(D')$.  In the latter case, observe that  in the event that
  wire names need to be permuted,  $w(D')$ will still 
  have a  subsequence $i,i,j,j$, although now $j$ denotes a different wire.  Note that such a step preserves the part of 
  the associated word to the left of the letter $i_v$.
  }
\end{remark}

\begin{proposition}
The word map $w$ 
gives a bijection from  the valid wire diagrams $D$  with $n$ wires to the permutations 
of alphabet $$\{ 1_p,2_p,\dots ,n_p,1_v,2_v,\dots ,n_v \} $$ with the requirements that $i_p $ appears to the left of $j_p$ for each $i<j$ and that $i_p$ appears to the left of $i_v$ for each $i$.  
\end{proposition}

\begin{proof}
The proof follows directly from observing that exactly these words arise and that the word map is 
invertible.
\end{proof}

\begin{corollary}\label{size-of-uncrossing-poset}
 The number
of elements in the uncrossing poset $P_n$ 
is  $$|P_n| = 1 + \frac{(2n)!}{n!2^n} = 1 + (2n-1)\cdot (2n-3)\cdot \cdots \cdot 3\cdot 1 .$$ 
\end{corollary}

\begin{definition}\label{start-set-def}
Let us define the {\bf start set} of $D$, denoted $S(D)$,  of  a wire diagram $D$ with $n$ wires as the $n$-subset of $\{ 1,2,\dots ,2n \} $ recording the positions in the word $w(D)$ where letters 
 with subscript $p$ appear.
\end{definition}

For example, $w(D) = 12331244$ gives rise to $S(D) = \{ 1,2,3, 7 \} $ while $w(D') = 11223344$ yields 
$S(D') = \{ 1,3,5,7 \} $.

Let us define the {\bf permutation} of $D$, denoted $\pi (D)$, as the permutation in $S_n$ obtained  by taking the restriction of $w(D) $ to the alphabet $1_v, 2_v,\dots , n_v$ and suppressing subscripts to obtain $\pi(D)$ in one-line notation.

\begin{theorem}\label{Bruhat-map}
Each interval $[D_1,D_2]$ in $P_n^*$  satisfying $S(D_1) = S(D_2)$  is isomorphic to the  type $A$ Bruhat order interval $[\pi (D_1) , \pi  (D_2)]$. 
\end{theorem}

\begin{proof}
Notice that the wire crossings  in a wire diagram
$D$ exactly correspond to the non-inversions in the restriction of 
$w(D)$ to the alphabet $\{ 1_v , 2_v, \dots , n_v \} $, so that the  desired
isomorphism 
is given by sending 
$D$ to this restriction of $w(D)$, namely to a permutation in one line notation.  To see that this is indeed a poset isomorphism, observe that uncrossing a pair of wires by swapping some $i_v$ with $j_v$  corresponds to applying the reflection $(i,j)$ on the left to the permutation $\pi (D)$  in one line notation given by the restriction of $w(D)$ to the letters $1_v, 2_v,\dots ,n_v$.
One may use Lemma ~\ref{noncross-down} to see that such an uncrossing step creates a double crossing if and only if the number of non-inversions decreases by more than one.  Thus, our cover relations in $P_n^*$  are
exactly the cover relations of the type  A  Bruhat order, as is immediate from the characterization of
type A Bruhat order cover relations given in Theorem ~\ref{Bruhat-cover-char}.
\end{proof}

Define the order  $\le_{lex } $ on subsets
of size $n$  of $\{ 1,\dots , 2n\} $ 
by  $\{ i_1,\dots ,i_n \} \le_{lex } \{ j_1,\dots , j_n \} $ for
$i_1 < i_2 < \dots < i_n $ and $j_1 < j_2 < \dots < j_n $ if and only if  either 
$(i_1,i_2,\dots ,i_n)$ is a 
lexicographically smaller vector than $(j_1,j_2,\dots ,j_n)$ or the two vectors are equal.

\begin{proposition}\label{start-set-order} 
If  $D_1 < D_2$ in $P_n^*$,  then $S(D_1) \le_{lex}  S(D_2)$ for $\le_{lex } $ the above  order on 
$n$-subsets of $\{ 1,2,\dots , 2n \} $.  In other words, $D_1 < D_2$ implies 
that we have 
$i_s< j_s$ for some $s$ with $i_r = j_r$ for all $r<s$.
\end{proposition}
  
  \begin{proof}
  Observe that we have $S(D_1) = S(D_2)$ for  each cover relation $D_1\prec D_2$ in which a subsequence $i,j,i,j$ of 
  $w(D_1)$ is transformed to $i,j,j,i$ in $w(D_2)$, namely for each cover relation swapping
  some pair $i_v$ and $j_v$ in the associated words.  Now we turn to the other type of wire 
  uncrossing discussed in Remark ~\ref{two-types-uncrossings}.  Observe 
  that  replacing $i,j,i,j$ in $w(D_1)$ by $i,i,j,j$ in $w(D_2)$ and then permuting the names of the wires so that wire names are first encountered in ascending order will  cause $S(D_2)$ to be obtained from 
  $S(D_1)$ by increasing the value of a single element of $S(D_1)$, namely replacing the 
    position of the first copy of $j$  in $w(D_1)$  by the larger position of the second copy of $i$ in $w(D_1)$.  That is, $S(D_2)$ lacks the position of the first copy of $j$ in $w(D_1)$ but instead has 
    the position of the larger copy of $i$ in $w(D_1)$, the latter of which is necessarily larger in order for the $i$ and $j$ wires to cross each other in $D_1$.
  \end{proof}
  
  The proof of Proposition ~\ref{start-set-order} also yields:
  
  \begin{corollary}\label{start-set-cor}
  Given $D_1 < D_2$ in $P_n^*$
   with $S(D_1) = S(D_2)$,  then all saturated chains from $D_1$ to $D_2$ consist of 
   uncrossing steps  which each replace some  $i,j,i,j$  in $w(D_1)$ by $i,j,j,i$ in $w(D_2)$.
   Each $D_1 < D_2$ in $P_n^*$ with $S(D_1) <_{lex} S(D_2)$ has the property that 
              all saturated chains from $D_1$ to $D_2$ 
       must use one or more
  uncrossings of the type which moves $i_v$ in $w(D_1)$ to the position in $w(D_2)$ that is 
  occupied by $j_p$ in $w(D_1)$; in this  case, $w(D_1)$ will have a 
  subsequence $i,j,i,j$ and $w(D_2)$ will have a subsequence $i,i,j,j$.
  \end{corollary}

Next we introduce for wire diagrams
a more refined  analogue  of  the  idea of 
inversion pairs of a permutation.  

\begin{definition}\label{noncross-def}
{\rm The {\bf noncrossing pair set of $D$}, denoted $N(D)$, of a wire diagram $D$ equals 
$N_1(D)\cup N_2(D)$ for  the disjoint sets $N_1(D)$ and $N_2(D)$ of ordered pairs defined as follows.  $N_1(D)$ consists of those ordered pairs $(i,j)$ for $i<j $ such that
$w(D)$ includes subexpression $i,j,j,i$.
$N_2(D)$ consists of those ordered pairs $(j,i)$ for $i<j$ such that 
 $w(D)$ instead has  subsequence $i,i,j,j$.}  
 \end{definition}

\begin{remark}
{\rm
The  $i,j,j,i$ and $i,i,j,j$ subsequence requirements for
$w(D)$ above which define  $N_1(D)$ and $N_2(D)$, respectively,  reflect   exactly 
 the two different  possible ways a pair of 
 wires $i$ and $j$ may be noncrossing.  Likewise having 
 the subsequence  $i,j,i,j$ in $w(D)$ encodes 
 combinatorially  exactly the condition that a pair of wires $i$ and $j$ cross
 each other.
 }
 \end{remark}

\subsection{Dual EC-shelling  for the uncrossing poset $P_n$}  

Let us now describe an edge labeling for $P_n^*$  whose well-definedness will follow immediately from Lemma ~\ref{noncross-down} 
 and  which we will prove is an EC-labeling in Theorem ~\ref{EL-shell-proof}  (based on a series of technical lemmas comprising 
 Section ~\ref{key-lemma-section}).
  
\begin{definition}\label{edge-label-def}
{\rm 
Label   $D \prec D'$ in $P_n^*$  as follows. 
If $w(D)$ has subsequence $k,m,k,m$, and we uncross wires $k$ and $m$ for $k<m$, 
to get $D'$ with 
$w(D')$ having subsequence $k,m,m,k$, then  let $\lambda (D,D') = (k,m)$.  (In this case, we have $(k,m)\in N_2(D')$, and  the $k$ wire ``turns right" upon approaching  the point where the wires previously crossed, to avoid crossing, assuming that this approach of the crossing is from a starting point that is the earlier of the two $k$ endpoints within $w(D)$). If   
$w(D')$ instead has subsequence $k,k,m,m$, then let  $\lambda (D,D') = (m,k)$.   (In this case, we
have $(m,k)\in N_2(D')$, and   the 
$m$ wire ``turns right"
upon approaching the previous wire crossing point, now using as the starting point of the approach the later of the two endpoints labelled $m$ in $w(D)$).  
Finally, we use a special symbol $L$ to label the remaining cover relations with a label  $L$ that is defined to 
be larger than all labels $(i,j)$ for $i<j $ and 
smaller than all labels $(r,s)$ for $r>s$;  specifically,  let $\lambda (D,\hat{1}) = L$ for each coatom $D\in P_n^*$.
} 
\end{definition}

The labels are ordered as follows, denoting by $<_{\lambda }$ this label order.

\begin{definition}\label{label-order}
{\rm 
The 
ordered pairs $(i,j)$ with $i<j$ are ordered amongst themselves 
 lexicographically, namely with the order 
 $(1,2) <_{\lambda }  (1,3) <_{\lambda }  (1,4) <_{\lambda }  \cdots <_{\lambda }  (1,n) 
<_{\lambda }  (2,3) <_{\lambda }  \cdots <_{\lambda }  (2,n) <_{\lambda }  \cdots <_{\lambda } 
 (n-1,n)$.  The ordered pairs $(r,s)$ for $r>s$ are ordered amongst themselves reverse 
linearly based on the second coordinate, breaking ties with reverse 
linear order on the first coordinate, so as
$(n,n-1) <_{\lambda }  (n,n-2) <_{\lambda }  (n-1,n-2) <_{\lambda }  (n,n-3) <_{\lambda }  (n-1,n-3) <_{\lambda }  (n-2,n-3) <_{\lambda }  \cdots <_{\lambda }  (n,1) <_{\lambda }  (n-1,1) <_{\lambda }  \cdots <_{\lambda }  (2,1)$.  
Finally,  $(i,j) <_{\lambda } L  <_{\lambda } (r,s)$ for each $i<j$ and each  $r>s$.
}
\end{definition}

\begin{remark}
{\rm 
The restriction of   $<_{\lambda }$  
to  labels  $(i,j)$ for $i<j$ coincides with the type A lexicographic
   reflection order (see Definition ~\ref{standard-reflection-orders}) 
   based on the ordering on simple roots induced by the ordering
    $s_1 < s_2 < \cdots < s_{n-1} $ on the corresponding  
    type A simple reflections.
    The label
 $L$ is  set to be 
  larger than all  these labels and smaller than all other labels.
    
  These choices will allow the transfer of some established results  from \cite{Dy93}
  related to shellability of
  Bruhat order to provide 
  useful ingredients to our proof that $\lambda $ is  an EC-labeling for $P_n^*$.
   }
\end{remark}

  Next is an analogue to a  property of inversions and Bruhat order, namely a characterization of
cover relations that will be useful later.

\begin{lemma}\label{noncross-down}
For $D$ with at least one pair of crossing wires, 
the cover relations $D'\prec D$ downward from $D$ in $P_n$ 
are given by exactly those wire 
uncrossings which get labeled via Definition ~\ref{edge-label-def}  by
 ordered pairs $(k,m)\not\in N(D)$ such that 
 the following conditions  met:
\begin{enumerate}
\item $(m,k)\not \in N(D)$
\item
If $k < m$, then  for each $l$ satisfying $k<l<m$ we have $$|\{ (k,l) , (l,m)\} \cap N(D)| = 1.$$
\item
If $k>m$, then for each  $l$ satisfying $l<m$ or $k<l$ we have $$| \{ (k,l), (l,m)\} \cap N(D)| = 1.$$
\end{enumerate}
\end{lemma}

\begin{proof}
The point is to observe  that the above combinatorial condition on $N(D)$ translates 
exactly  to the no-double-crossing condition for the diagram $D'$ obtained by 
performing the uncrossing of wires $k$ and $m$ in the way that is 
dictated by the label  $\lambda (D',D) = (k,m)$ given by Definition ~\ref{edge-label-def}.
That is, we use the label $(k,m)$ to dictate the nature of the uncrossing of wires and  will
show that the above condition describes when this indeed gives a cover relation.
  
The equivalence of this reformulation to the no-double-crossing condition
can be checked by a
straightforward consideration of the various cases given by the various 
 words consisting of the 
letters $k,k,l,l,m,m$ in those orders which may appear as subsequences of $w(D)$ for 
$k<m$ and then separately for $k>m$; it is important to  utilize our assumption that we have 
either $k<l<m$ or $l<m<k$ or $m<k<l$ to restrict which subsequences need to be considered.  
 In other words, we must  consider the various allowable ways 
these three wires may cross each other or avoid crossing each other under our
hypotheses.   It may help the reader 
to draw a picture and calculate the contribution of wires $k,l,m$ to $N(D)$ for the various 
allowable subsequences of $w(D)$ comprised of the multiset of letters $\{ k,k,l,l,m,m \} $. 
\end{proof}

\begin{lemma}\label{Bruhat-isom}
Given $D_1 < D_2$ with $S(D_1) = S(D_2)$, then the restriction of $\lambda $ to the interval $[D_1,D_2]$ in $P_n^*$  with label ordering  $<_{\lambda }$  is exactly the Dyer reflection order EL-labeling for  type A Bruhat order resulting from the lexicographic reflection order given by the ordering $$(1,2) < (1,3) < \dots < (1,n) < (2,3)  < \cdots < (2,n) < \cdots < (n-1, n)$$ on the type A positive roots.
\end{lemma}

\begin{proof}
This is immediate from the definition of our labeling together with our earlier isomorphism in Theorem  ~\ref{Bruhat-map}  which  maps an allowable  uncrossing $D\prec D'$ 
of a pair of wires $i$ and $j$, namely one with $S(D) = S(D')$, 
 to the application of the reflection $(i,j)$ to the corresponding element of Bruhat order.
\end{proof} 

See Definition  ~\ref{ec-def} for the notions of EC-shellability, topological ascent and topological descent, used heavily in  the remainder of this section and all of the next section. 

\begin{theorem}\label{EL-shell-proof}
$P_n^*$  is  EC-shellable via edge labeling $\lambda $ (see Definition ~\ref{edge-label-def}) 
for $P_n^*$ with respect to the  ordering $<_{\lambda }$ (see Definition ~\ref{label-order})
on edge labels.   Therefore, $P_n$ is shellable.
\end{theorem}

\begin{proof}
First note that shellability of $P_n^*$  will imply  shellability of $P_n$ since these posets  have the 
same chains and hence the same order complex as each other.

To prove EC-shellability of $P_n^*$, we need to prove for any $u<v$ there is a unique topologically ascending  saturated chain from $u$ to $v$.  As a word of caution, when we leave off the adjective ``topologically'' below, this is deliberate, and we really do mean traditional ascents and descents rather than topological ones in that case.  

Lemma ~\ref{semi-mod} proves for $u < v<\hat{1} $ in $P^*$ (in other words for $u > v > \hat{0}$ in $P$) 
that there is a unique saturated chain from $u$ to $v$ not having any topological descents, which
therefore must be the lexicographically first one.

For $v=\hat 1$ we need a separate argument:
Lemmas ~\ref{0-exists-first} and ~\ref{0-unique} prove 
that the lexicographically first 
saturated chain from $u$ to  $\hat 1$ 
has weakly ascending labels and that every other saturated chain  from 
$u$ to $\hat 1$ has at least one descent.  Lemma ~\ref{descent-is-topol-descent} proves that each 
descent $\lambda (x,y) > \lambda (y,z)$ for $z\ne \hat{1} $ is a topological descent, implying that 
each saturated chain with such a descent has a topological descent; moreover, 
any descent $\lambda (x,y) > \lambda (y,z)$ for $z=\hat{1} $  is a topological descent because the wire diagram
$x$ then has a single crossing with $\lambda (x,y) = (j,i) > (i,j) = \lambda (x,y')$ for $i<j$ the two wires comprising the unique wire crossing in $x$.
Thus, the lexicographically first saturated chain from $u$ to $\hat 1$ is the only topologically ascending chain. 
\end{proof}

\begin{corollary}\label{CW-proof}
The uncrossing order $P_n$ is a CW poset.
\end{corollary}

\begin{proof}
Our proof of Lam's shellability conjecture given 
 in Theorem ~\ref{EL-shell-proof}   will imply that 
uncrossing posets are CW posets, due to the fact that  they are by definition 
graded posets (see Remark ~\ref{rank-def}) 
and were already proven to be  Eulerian in \cite{La14a}.  Thus, 
Theorem ~\ref{Danaraj-Klee} applies.
\end{proof}


\section{Key Technical Lemmas}\label{key-lemma-section}

Now we turn to the heart of the proof that our labeling $\lambda $ for $P^*_n$ is an EC-labeling, namely  the lemmas which together provide the technical details of the proof.   First we handle intervals $[u,v]$ for $v < \hat{1}$, proving each such interval has a unique topologically ascending chain.  We begin by showing how traditional descents  contained within such intervals always yield topological descents, thereby considerably restricting the possibilities for how topologically ascending chains may arise.  

\begin{lemma}\label{descent-is-topol-descent}
For each  $u < v < \hat{1}$ in $P_n^*$, any descent in any saturated chain from $u$ to $v$ (with respect to the edge labeling $\lambda $) is a 
topological descent.
\end{lemma}

\begin{proof}
It suffices to prove the following: 
given $x\prec y \prec z$ in $P_n^*$ with labels $\lambda (x,y) = (p,q)$ and 
$\lambda (y,z) = (r,s) $ such that $(p,q) >_{\lambda } (r,s)$,  then there is a  saturated chain $x\prec y'\prec z$ with 
lexicographically smaller label sequence from $x$ to $z$.   We break the proof of 
this assertion  into cases, based on the various ways a descent $\lambda (x,y) >_{\lambda }
\lambda (y,z)$  may arise.  

First suppose
there are  four different wires involved in the  two consecutive wire uncrossings $x\prec y \prec z$  in $P_n^*$ comprising a 
descent $\lambda (x, y) >_{\lambda }  \lambda (y, z)$. Notice that these  two uncrossings may  be carried out in the other order
yielding some $x\prec y' \prec z$ in $P_n^*$, since reversing the order in which the two uncrossings are carried out  will not impact the fact that $z$ has exactly two fewer crossings than $x$, forcing $y'$ to have exactly one more crossing than $z$ and one fewer crossing than $x$.  
Reversing the order in which these two uncrossings are carried out preserves both the wire name at the earlier of the two endpoints for  the  smallest of the four wires involved in the two uncrossings   as well as
preserving the property that this endpoint belongs to the smallest of the four wires involved in the two uncrossings.
Letting $\lambda (x,y') = (a,b)$ and $\lambda (y',z) = (c,d)$, we claim  that we have $p<q$ if and only if we have $c<d$ and likewise we have $r<s$ if and only if we have $a<b$; these observations follow from the fact that deleting other wires not involved in the uncrossings  being performed  does not impact which of these wires have endpoints that are encountered first in clockwise order proceeding from our basepoint.


These observations together with our label ordering and the fact that the pair of uncrossing steps uses four distinct wires  will yield $(a,b) <_{\lambda } (p,q)$ from $(p,q) >_{\lambda } (r,s)$, just as needed, as we now check by running through the various possible  cases. 
 The case with $p>q$ and $r>s$ must have $q<s$ in order for $x\prec y \prec z$
 to have a descent $(p,q) = \lambda (x,y) >_{\lambda }  \lambda (y,z)  = (r,s)$ 
 in its labels.
 Hence, such a descent   must have $q$ as the 
overall smallest wire amongst the four wires involved in the two uncrossing steps. 
This yields the result in this case whether we have $a<b$ (which implies  
$(a,b) <_{\lambda } (p,q)$ due to 
having $a<b$ and $p>q$)
 or we have $a>b$ (since in this case we have $b>q$ with $a>b$ and $p>q$, hence 
$(a,b) <_{\lambda } (p,q)$).  See Figure 4.
\begin{figure}[htb]
\begin{center}
\includegraphics[height=1.5in,width=1.5in,angle=0]{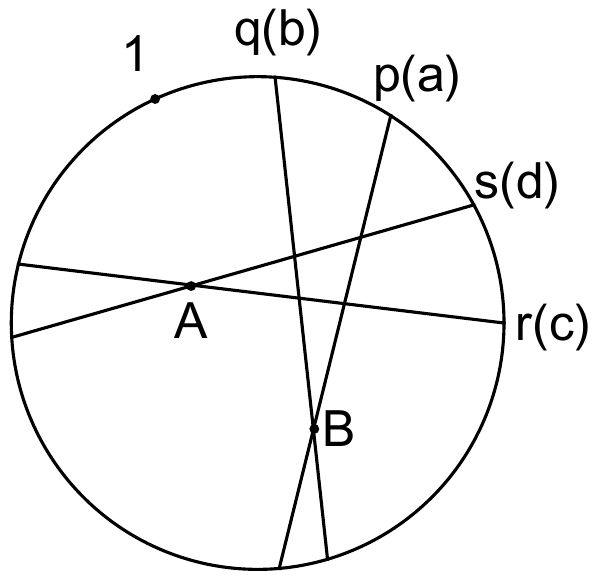}
\caption{$u\prec z$ and $x\prec y'$ uncrossings at $A$;  $x\prec y$ and $y'\prec z$ uncrossings at $B$}
\end{center}
\end{figure}
This same analysis also applies in the case with 
$p>q$ and $r<s$ in the event that  we also have $r>q$.  
If we instead have 
$p>q$ and $r<s$ with $r<q$, then this implies  $a<b$ with  $r=a$, by our observations above,
 yielding the result.
   See Figure 5.
    Finally, for  $p<q$, then having a descent in $x\prec y\prec z$ means 
  we also must have $r<s$ with $r<p$, which implies $a<b$ with  $r=a$, 
  giving the result in this case. 
See Figure 6.
This completes the proof for all possible  cases  with 
four different wires involved in two consecutive   wire 
uncrossings carried out by cover relations
$x\prec y \prec z$.

   \begin{figure}[htb]
\begin{center}
\includegraphics[height=1.5in,width=1.5in,angle=0]{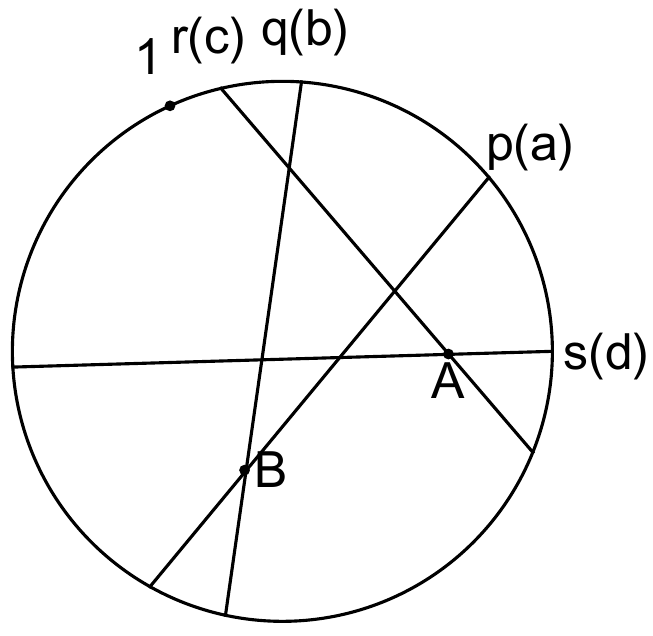}
\caption{$x\prec y'$ and $y\prec z$ uncrossings at $A$; 
 $x\prec y$ and $y'\prec z$ uncrossings at $B$}
\end{center}
\end{figure}  

\begin{figure}[htb]
\begin{center}
\includegraphics[height=1.5in,width=1.5in,angle=0]{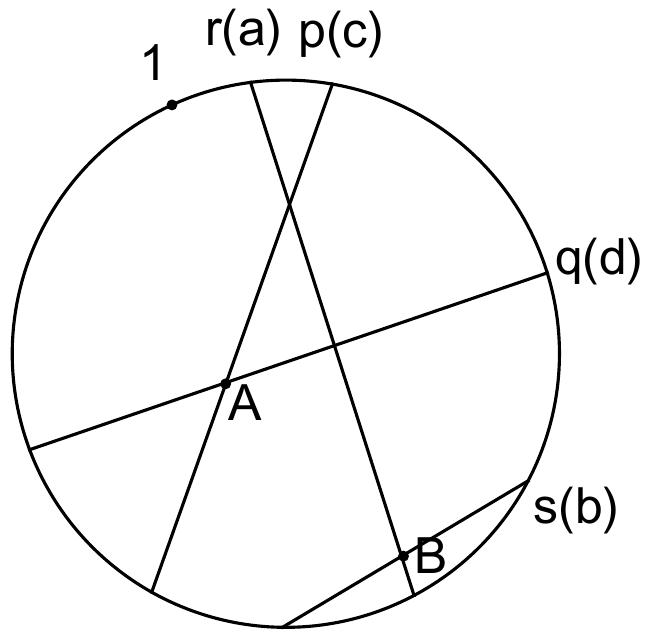}
\caption{$x\prec y$ and $y'\prec z$ uncrossings at $A$; $x\prec y'$ and $y\prec z$ uncrossings at $B$}
\end{center}
\end{figure}

Now to $x\prec y \prec z$ carrying out two uncrossings involving a total of three wires.  
All of the possible
cases in $P_n^*$ correspond naturally (by restriction to these
three wires)  to  cases  that  arise in $P_3^*$. 
This description of various cases  involving three wires according to how they restrict to $P_3^*$ 
 seems to be a good way to organize these cases  for $P_n^*$.    
 We will prove that each such descent in $P_n^*$  restricts to 
a descent  in $P_3^*$.    The authors have checked by hand 
that all descents in $P_3^*$  are  topological descents.  
See Figure 7 for this edge labeling for 
$P_3^*$, from 
which the interested reader may also  
check this claim quite easily for $P_3^*$; it is important to note that one must traverse 
the cover relations downward rather than upward  in Figure 7
so as to consider saturated chains in $P_3^*$ rather than in $P_3$.   We will also 
prove for the inclusion map from $P_3^*$ to $P_n^*$ that is inverse to the aforementioned 
restriction map 
 that each  topological descent  in $P_3^*$ includes into $P_n^*$ as a 
topological descent in $P_n^*$.  Once these claims are proven, 
this will  yield the desired result.

  \begin{figure}[h]\label{poset-example}
        \begin{picture}(325,500)(-80,-25)
 
       \put(77,441){\line(1,0){36}}
      \put(80,462){\line(1,-1){27}}
      \put(83,434){\line(1,1){28}}

        \put(83,434){\circle*{4}}
              \put(113,441){\circle*{4}}
       \put(77,441){\circle*{4}}
       \put(108,434){\circle*{4}}
       \put(80,462){\circle*{6}}
       \put(111,462){\circle*{4}}

 \put(70,465){$1$}
 \put(115,462){$2$}
 \put(118,439){$3$}
 \put(112,425){$1$}
 \put(72,425){$2$}
 \put(66,439){$3$}

\put(40,445){$\hat{1} = $}
       \put(95,450){\circle{100}}
       
       \curve(5,392,  80,420)
       \curve(95,392,  95,420)
       \curve(180,392,  110,420)
       
       \put(5, 410){$(1,2)$}
       \put(98, 400){$(2,3)$}
       \put(155, 410){$(3,1)$}

               \put(95,365){\circle{100}}             
      \put(80,377){\line(1,-1){27}}
\curve(77, 356, 94, 362, 111,377)
\curve(83,349, 98, 356, 113,356)

        \put(83,349){\circle*{4}}
              \put(113,356){\circle*{4}}
       \put(77,356){\circle*{4}}
       \put(108,349){\circle*{4}}
       \put(80,377){\circle*{6}}
       \put(111,377){\circle*{4}}
       
       \put(70, 379){$1$}

        \put(-5,365){\circle{100}}
\curve(-17, 349,  -15, 363, -20, 377)
\curve(8,349, 5, 363,  11, 377)
\curve(-23,356,  13, 356)

        \put(-17,349){\circle*{4}}
              \put(13,356){\circle*{4}}
       \put(-23,356){\circle*{4}}
       \put(8,349){\circle*{4}}
       \put(-20,377){\circle*{6}}
       \put(11,377){\circle*{4}}

 \put(-30, 379){$1$}
 
        \put(195,365){\circle{100}}
      \put(183,349){\line(1,1){28}}
\curve(180,377, 197, 363, 213, 356)
\curve(177,356,  192, 357, 208, 349)

        \put(183,349){\circle*{4}}
              \put(213,356){\circle*{4}}
       \put(177,356){\circle*{4}}
       \put(208,349){\circle*{4}}
       \put(180,377){\circle*{6}}
       \put(211,377){\circle*{4}}

\put(170,379){$1$}
        
        \curve(-15,335,  -45,250)
        \curve(-5,335,  5,250)
        \curve(5,335, 60,250)
        \curve(15,335, 115,250)
        
        \put(-57,295){$(1,3)$}
        \put(-26, 290){$(2,3)$}
        \put(26, 256){$(3,2)$}
        \put(52, 273){$(3,1)$}

        \put(41, 325){$(1,2)$}
        \put(65, 301){$(1,3)$}
        \put(93, 306){$(3,1)$}
        \put(122, 328){$(2,1)$}
        
        \put(86, 286){$(1,2)$}
        \put(110, 264){$(3,2)$}
        \put(188, 290){$(2,3)$}
        \put(222, 295){$(2,1)$}

        \curve(80,335, -35, 250)
        \curve(90,335,  10, 250)
        \curve(100,335,  170, 250)
        \curve(110,335,  230, 250)
        
        \curve(175, 335, 70,250) 
        \curve(185, 335, 125, 250) 
        \curve(195, 335, 180, 250)
        \curve(205, 335, 240, 250)

\put(65,220){\circle{100}}
\curve(81,232, 78, 221, 83, 211)
\curve(50,232, 56, 218, 53,204)
\curve(47,211, 62, 213, 78,204)

        \put(53,204){\circle*{4}}
              \put(83,211){\circle*{4}}
       \put(47,211){\circle*{4}}
       \put(78,204){\circle*{4}}
       \put(50,233){\circle*{6}}
       \put(81,232){\circle*{4}}

\put(40, 235){$1$}

\put(125,220){\circle{100}}
\curve(110,232, 127, 217, 143,211)
\curve(107,211, 114,211, 113,204)
\curve(141,232, 135, 218, 138,204)

        \put(113,204){\circle*{4}}
              \put(143,211){\circle*{4}}
       \put(107,211){\circle*{4}}
       \put(138,204){\circle*{4}}
       \put(110,233){\circle*{6}}
       \put(141,232){\circle*{4}}

\put(100, 235){$1$}

\put(5,220){\circle{100}}
\curve(-13,211, 4, 216,  21, 232)
\curve(23,211, 17, 210, 18, 204)
\curve(-10, 232, -4, 218,  -7, 204)

        \put(-7,204){\circle*{4}}
              \put(23,211){\circle*{4}}
       \put(-13,211){\circle*{4}}
       \put(18,204){\circle*{4}}
       \put(-10,233){\circle*{6}}
       \put(21,232){\circle*{4}}

\put(-20,235){$1$}

\put(-55,220){\circle{100}}
\curve(-73,211, -67, 221, -70,232)
\curve(-67,204, -51, 214, -37,211)
\curve(-39,232, -45, 218, -42,204)

        \put(-67,204){\circle*{4}}
              \put(-37,211){\circle*{4}}
       \put(-73,211){\circle*{4}}
       \put(-42,204){\circle*{4}}
       \put(-70,233){\circle*{6}}
       \put(-39,232){\circle*{4}}

\put(-80,235){$1$}

\put(185,220){\circle{100}}
\curve(170,232, 187, 217,  203,211)
\curve(173,204, 185, 207, 198,204)
\curve(201,232,  184, 217, 167,211)

        \put(173,204){\circle*{4}}
              \put(203,211){\circle*{4}}
       \put(167,211){\circle*{4}}
       \put(198,204){\circle*{4}}
       \put(170,233){\circle*{6}}
       \put(201,232){\circle*{4}}

\put(160,235){$1$}

\put(245,220){\circle{100}}
\curve(230,232, 245, 228, 261,232)
\curve(227,211, 242, 211,  258,204)
\curve(233,204,  248, 211, 263,211)

        \put(233,204){\circle*{4}}
              \put(263,211){\circle*{4}}
       \put(227,211){\circle*{4}}
       \put(258,204){\circle*{4}}
       \put(230,232){\circle*{6}}
       \put(261,232){\circle*{4}}

\put(220,232){$1$}

\curve(-55, 190,  -35, 105)
\curve(-45, 190,  25, 105)
\put(-75, 150){$(2,3)$}
\put(-34, 177){$(3,2)$}

\curve(0, 190, -25, 105)
\curve(10, 190, 90, 105)
\put(-21, 115){$(1,2)$}
\put(5, 160){$(2,1)$}

\curve(60, 190, 35, 105)
\curve(70, 190, 145, 105)
\put(16,130){$(1,3)$}
\put(65,160){$(3,1)$}

\curve(120, 190, 100, 105)
\curve(130, 190, 155, 105)
\put(91,175){$(2,1)$}
\put(142,148){$(1,2)$}

\curve(180, 190, 45, 105)
\curve(190, 190, 210, 105)
\put(67,137){$(1,2)$}
\put(192,181){$(2,1)$}

\curve(240, 190, 115, 105)
\curve(250, 190, 220, 105)
\put(161,128){$(3,2)$}
\put(233,130){$(2,3)$}

\put(95,75){\circle{100}}
\curve(80,87, 95, 83, 111,87)
\curve(77,66, 83,66,  83,59)
\curve(108,59, 108,66, 113,66)

        \put(83,59){\circle*{4}}
              \put(113,66){\circle*{4}}
       \put(77,66){\circle*{4}}
       \put(108,59){\circle*{4}}
       \put(80,88){\circle*{6}}
       \put(111,87){\circle*{4}}

\put(70,90){$1$}

\put(155,75){\circle{100}}
\curve(140,87, 168,59)
\curve(171,87, 168, 76, 173,66)
\curve(143,59, 143,66, 137,66)

        \put(143,59){\circle*{4}}
              \put(173,66){\circle*{4}}
       \put(137,66){\circle*{4}}
       \put(168,59){\circle*{4}}
       \put(140,88){\circle*{6}}
       \put(171,87){\circle*{4}}

\put(130,90){$1$}

\put(35,75){\circle{100}}
\curve(20,87, 23, 76, 17,66)
\curve(23,59, 35, 63, 48,59)
\curve(53,66, 48, 76,  51,87)

        \put(23,59){\circle*{4}}
              \put(53,66){\circle*{4}}
       \put(17,66){\circle*{4}}
       \put(48,59){\circle*{4}}
       \put(20,88){\circle*{6}}
       \put(51,87){\circle*{4}}

\put(10,90){$1$}

\put(-25,75){\circle{100}}
\curve(-43,66,  -37, 76, -40,87)
\curve(-37,59, -9,87)
\curve(-7,66, -13,66, -12,59)

        \put(-37,59){\circle*{4}}
              \put(-7,66){\circle*{4}}
       \put(-43,66){\circle*{4}}
       \put(-12,59){\circle*{4}}
       \put(-40,88){\circle*{6}}
       \put(-9,87){\circle*{4}}

\put(-50,90){$1$}

\put(215,75){\circle{100}}
\curve(200,87, 215, 83,  231,87)
\curve(203,59, 215, 62, 228,59)
\curve(233,66, 215, 69, 197,66)

        \put(203,59){\circle*{4}}
              \put(233,66){\circle*{4}}
       \put(197,66){\circle*{4}}
       \put(228,59){\circle*{4}}
       \put(200,87){\circle*{6}}
       \put(231,87){\circle*{4}}

\put(190,89){$1$}

\curve(75, -5, -20, 47)
\curve(85, -5, 35, 47)
\curve(95,-5, 95, 47)
\curve(105,-5, 150, 47)
\curve(115, -5, 200, 47)

\put(-7,25){$L$}
\put(38,25){$L$}
\put(83,25){$L$}
\put(118,25){$L$}
\put(185,25){$L$}

\put(92,-25){$\hat{0}$}
          
        \end{picture}
        \caption{Dual $EC$ labeling for $P_3$} 
\end{figure}

 Consider  an edge label $(a,b)$ for an uncrossing in $P_n^*$ arising in the case of a 
 descent $x\prec y \prec z$ involving a total of three  wires in the two consecutive uncrossings.
Also consider the unique uncrossing  $x\prec y'$ for $y'\ne y$ and  $y'\prec z$, noting that 
$x\prec y'\prec z$ also carries out uncrossings involving only these same  
 three wires.   
  Let us show now that passing back and forth between  $P_3^*$ to $P_n^*$ by wire inclusion and by restriction to these three wires, respectively,  will not impact the relative order of the labels $\lambda (x,y)$ and $\lambda (x,y')$.  For convenience in doing this, let us
 denote by $(a',b')$ the corresponding 
 edge label for $P_3^*$ obtained by  restriction to these three wires.
  This desired result will 
follow directly from the following  three facts that are themselves
 immediate from the definitions of the labels for uncrossing steps and of the 
 label ordering $<_{\lambda }$:
\begin{enumerate}
\item 
A label $(a,b)$ 
has $a < b$ (resp. $a>b$)  if and only if  the label $(a',b')$ has $a'<b'$ (resp. $a'>b'$).
\item
Two labels $(a,b)$ and $(c,d)$ in $P_n^*$ for uncrossings involving 
  a total of three wires that either occur in consecutive steps $x\prec y\prec z$ or in steps 
  $x\prec y$ and $x\prec y'$ will 
satisfy $\min \{ a,b\} < \min \{ c,d\} $
if and only if the  labels for the corresponding uncrossings  in $P_3^*$ satisfy 
$\min \{ a',b' \} < \min \{ c',d'\} $.
\item 
For $\lambda (x,y) = (a,b)$ and $\lambda (x,y') = (c,d)$, we have $\min \{ a,b\} = \min \{ c,d \} $
if and only if $\min \{ a',b'\} = \min \{ c',d'\} $.  In this case of equality, we also have 
$\max \{ a,b \} < \max \{ c,d\} $ if and only if $\max \{ a',b' \} < \max \{ c',d' \} $.
\end{enumerate}  

If we can show  that each scenario producing a descent 
$\lambda (x,y) >_{\lambda }  \lambda (y,z)$  in $P_n^*$ with three wires involved in the wire 
uncrossings   corresponds to a situation also giving a descent in $P_3^*$, we can use the above observations to deduce that each such descent in 
$P_n^*$ is a topological descent by the following chain of reasoning.  
Having a descent in  $P_n^*$ will restrict to one in $P_3^*$ which will then   imply there is a lexicographically 
earlier label sequence  from $x$ to $z$ in $P_3^*$.
 By  virtue of 
 the   preservation of relative order of  labels on $x\prec y$ and $x\prec y'$ upon restriction from 
 $P_n^*$ to $P_3^*$ and the inverse operation of inclusion of $P_3^*$ into $P_n^*$, a topological 
 descent in $P_3^*$ will correspond via wire inclusion to a topological descent in $P_n^*$.  That is,
   the lexicographically earlier label sequence in $P_3^*$ from $x$ to $z$  (guaranteed to exist in $P_3^*$ 
 by virtue of  $x\prec y\prec z$ being a topological descent in $P_3^*$) 
will imply  the existence of a corresponding  lexicographically earlier 
label sequence from $x$ to $z$ in $P_n^*$ by inclusion of $x\prec y'$ into $P_n^*$ by wire 
inclusion.  This will ensure that  $x\prec y\prec z$ will be  a 
topological descent in $P_n^*$.

Now to the claim about descents in $P_n^*$ restricting to descents in $P_3^*$ for $x\prec y\prec z$
with uncrossings involving a total of three wires.  
Suppose we have label $\lambda (x,y) = (r,s)$ and then $\lambda (y,z) = (p,q)$ 
for $(r,s) >_{\lambda } (p,q)$ in $P_n^*$. 
If we have   $r<s$, then the 
uncrossing step with label $(r,s)$  renames only  the later  endpoints (in clockwise order proceeding from basepoint)  of the   wires being uncrossed.  But we must have $p<q$ in this case in order  to have a descent and also must have 
$p\le r$; these 
 properties are not 
impacted  in passing from $p,q,r,s$ to $p',q',r',s'$, so the descent stays a descent upon 
restriction to $P_3^*$, completing
the $r<s$ case.
 Now suppose 
 $r>s$.  If we have a descent comprised of $\lambda (x,y) = (r,s)$ and
 $\lambda (y,z) = (p,q)$ for $p<q$ in $P_n^*$,
  then  the corresponding consecutive labels $(r',s')$ and $(p',q')$ in $P_3^*$ also comprise a
  descent in $P_3^*$, 
  since $r>s$ implies $r'>s'$ and $p<q$ implies $p'<q'$  in this case.
   Likewise, for $r>s$ and $p>q$ with $q>s$, restricting  from $P_n^*$ to $P_3^*$ 
    will yield $r'>s', p'>q'$ and $q'>s'$ by the three observations listed earlier in this proof.
      Finally, observe that it is not possible to have consecutive labels $\lambda (x,y) = (j,i) $ and then 
  $\lambda (y,z) = (k,i)$ for $i<j<k$ in any saturated chain in $P_n^*$, since the uncrossing step $x\prec y$ given by $(j,i)$ will cause the wires $i$ and $k$ no longer  to cross each other, rendering
  $y\prec z$ with label  $\lambda (y,z) = (k,i)$ impossible.  This completes the $r>s$ case, and hence completes the   case  
  in which a total of  three wires are involved in the two 
  consecutive uncrossing steps comprising a descent in $P_n^*$.
 \end{proof} 
  
The next lemma will complete the proof that each interval $[u,v]$ for $v<\hat{1}$ has a unique topologically ascending chain from $u$ to $v$, with this saturated chain being the lexiocographically earliest saturated chain from $u$ to $v$.  
  
  \begin{lemma}\label{semi-mod}
Given $u< v < \hat{1} $ in $P_n^*$, there is a unique saturated chain from $u$ to $v$ with no topological descents (with respect to edge labeling $\lambda $).
\end{lemma}

\begin{proof}
Let $C$ be a saturated chain from $u$ to $v$  which has no topological descents.   At least one such saturated chain exists, since 
the lexicographically first saturated chain from $u$ to $v$ takes this form.  Since Lemma 
~\ref{descent-is-topol-descent} proved that each descent is a topological descent, we are assured 
that the labels for $C$ are weakly ascending.  For $C$ comprised of cover relations 
$u = v_0 \prec v_1 \prec v_2 \prec \cdots \prec v_r \prec v_{r+1} = v$, this means we have 
$$\lambda (u,v_1) \le \lambda (v_1,v_2)\le \cdots \le \lambda (v_r, v).$$
By definition of our label ordering $<_{\lambda }$, 
all of the labels on $C$  of the form $(i,j) $ for $i<j$ must occur 
 lower on the saturated chain  than all of its labels 
which are of the form $(j',i')$ for $j'>i'$.  The labels of the form $(i,j)$ with $i<j$ must proceed upward in $C$  from smallest to largest value of $i$, breaking ties by proceeding from smallest to largest value of $j$.  Otherwise, we would have a descent, and hence by Lemma ~\ref{descent-is-topol-descent} we would have a topological descent.  
Likewise, 
the labels of the form $(j',i')$ for $j'>i'$ must proceed upward in $C$ from largest to smallest value of $i'$, breaking ties by proceeding from largest to smallest value of $j'$.

Our plan is to show that such a saturated chain $C$ from $u$ to $v$ comprised entirely of topological ascents
is uniquely determined by the associated words  $w(u)$ and $w(v)$. 
When we have $S(u) = S(v)$, then 
the result follows  immediately from  Proposition ~\ref{start-set-order},   
 Corollary ~\ref{start-set-cor}, and Lemma ~\ref{Bruhat-isom} since these results 
 show in this case   that $[u,v]$ 
 is isomorphic to a type A Bruhat interval with our labeling restricting to a Bruhat order reflection order EL-labeling  for this interval (using Dyer's results in \cite{Dy93} guaranteeing this is indeed an EL-labeling for Bruhat order); in particular, this guarantees  there is a unique saturated chain  with weakly ascending labels in this case. 
Therefore, we may assume  henceforth  that  $S(u) \ne S(v)$. 
   In fact, it will suffice by this same reasoning to prove that
 the portion $v_m \prec v_{m+1} \prec  \cdots \prec v_r \prec v$ of $C$ having (1) $S(v_m) = S(u)$ and (2)  
 $S(v_{m+1})\ne S(u)$ is uniquely determined.  By definition of our labeling, this will be exactly the part of $C$ using labels $(j',i')$ with $j'>i'$.
 Now we turn to this task, in particular showing that such $v_r$ will be uniquely determined by $u$ and $v$.  Once we do that, the same argument may be applied repeatedly to determine $v_{r-1}$ then 
 $v_{r-2}$ and so on, until eventually reaching some $v_m$ with $S(v_m) = S(u)$.   At that point, $u\prec v_1\prec v_2\prec \cdots \prec v_m$
 is uniquely determined by our above argument handling the $S(u)=S(v)$ case. 

Let us begin by making and verifying  two observations to be used later.
The first observation 
is that proceeding up any cover relation $v_i\prec v_{i+1}$ anywhere  in $C$ either fixes all letters of the form $j_p$ (namely with subscript $p$) in passing from $w(v_i)$ to $w(v_{i+1})$ or else moves one or more such letters rightward while moving a single letter $i_v $ leftward.  The former  situation is what happens for each cover relation labeled $(k,l)$ for $k<l$, since such a cover relation exchanges  some $k_v$ with some $l_v$, 
while the latter describes cover relations labeled $(j,i)$ for $j>i$ which  move one or more letters of the form $j_p $ rightward while moving a single letter $i_v $ leftward to the position of the leftmost of these letters moving  rightward. 
In either case, it cannot happen that a letter $j_p$ moves leftward.  In particular, this implies that each letter $j_p$ whose position is the same in $w(u)$ and in $w(v)$ must be fixed throughout each saturated chain from $u$ to $v$.
Our second observation is 
that when a cover relation $v_i \prec v_{i+1}$ labeled $(j,l)$  for $j>l$ 
 moves a  letter $l_v$ 
leftward in passing from $w(v_i)$ to $w(v_{i+1})$ by moving the letter to a position that was occupied by some
$j_p$ in $w(v_i)$, then we claim that  any letter $k_v$ which appears to the left of $l_v$ in $w(v_i)$ but to the right of 
$l_v$ in $w(v_{i+1})$ must have $k>l$, as explained next.   Otherwise the cover relation $v_i \prec v_{i+1}$ would introduce a double crossing of the wires labeled $k$ and
$l$ in $w(v_{i+1})$  by virtue of $w(v_i)$  necessarily having the 
subsequence  $ k_p ,l_p, j_p, k_v, l_v , j_v$.  But this  would contradict $v_i\prec v_{i+1}$ 
being a cover relation, completing the proof of this claim.  See Figure 8.

\begin{figure}[htb]
\begin{center}
\includegraphics[width=2.5in,angle=0]{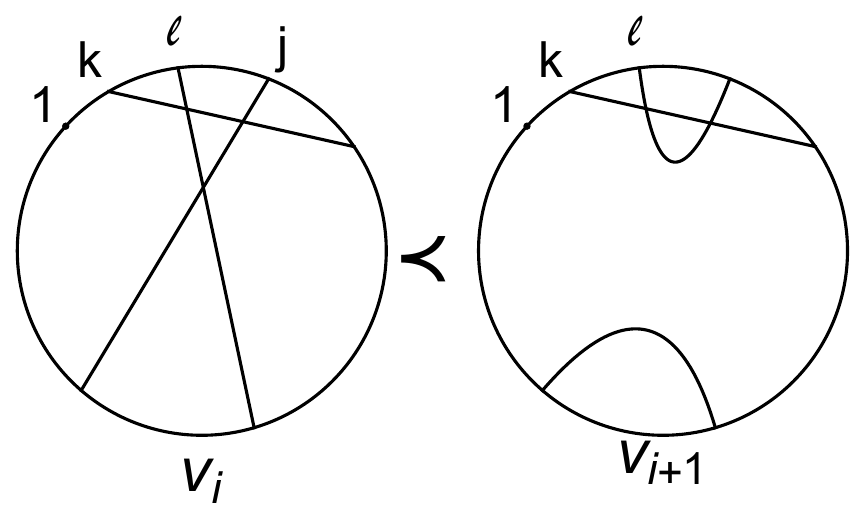}
\caption{$\lambda (v_i,v_{i+1}) = (j,l)$}
\end{center}
\end{figure}

Let us now show  that the cover relation $v_r\prec v $ in $C$ 
must have label $(j,l)$ for $j>l$  for a uniquely 
determined value $l$.  Specifically, we will show that $l$  must be as small as possible among letters $l_v$ which either  appear at a position  in $w(v)$ that is in  $S(u)$ or where $w(v)$ has a subsequence $l_p,m_p,m_v, l_v$ with  $m_v$ appearing 
at a position in $w(v)$ that is in $S(u)$; in the latter case, $l_v$ is 
then the right endpoint of a wire in $v$   having nested below it such a letter $m_v$.   See Figure 9.
If $l$ were not as small as possible with this property, then $C$ would necessarily  have a label 
$(j',l)$ for the same value $l$ and some  $j'>l$ at some point  lower in the saturated chain, since eventually the saturated chain must  move the   letter $j_p$  either to  the  location occupied by $l_v$ in $w(v)$ or to the nested  $m_v$ position described above where $j_p$  appears in  $w(v_m)$ in
that case.  But this  label $(j',l)$  lower on $C$  
will guarantee the existence of a larger label $(j',l)$ lower in the label sequence for $C$ than the 
label $ \lambda (v_r,v) = (j',l')$ with 
$j'>l'>l$ appearing at a higher position in $C$.  In particular, this ensures a descent (and hence a topological descent by Lemma
~\ref{descent-is-topol-descent}) somewhere in $C$, a contradiction, The upshot is that the smaller 
value  $l$ in the label $(j,l)$   is uniquely determined as described  just above. 

\begin{figure}[htb]
\begin{center}
\includegraphics[height=2.5in,width=3in,angle=0]{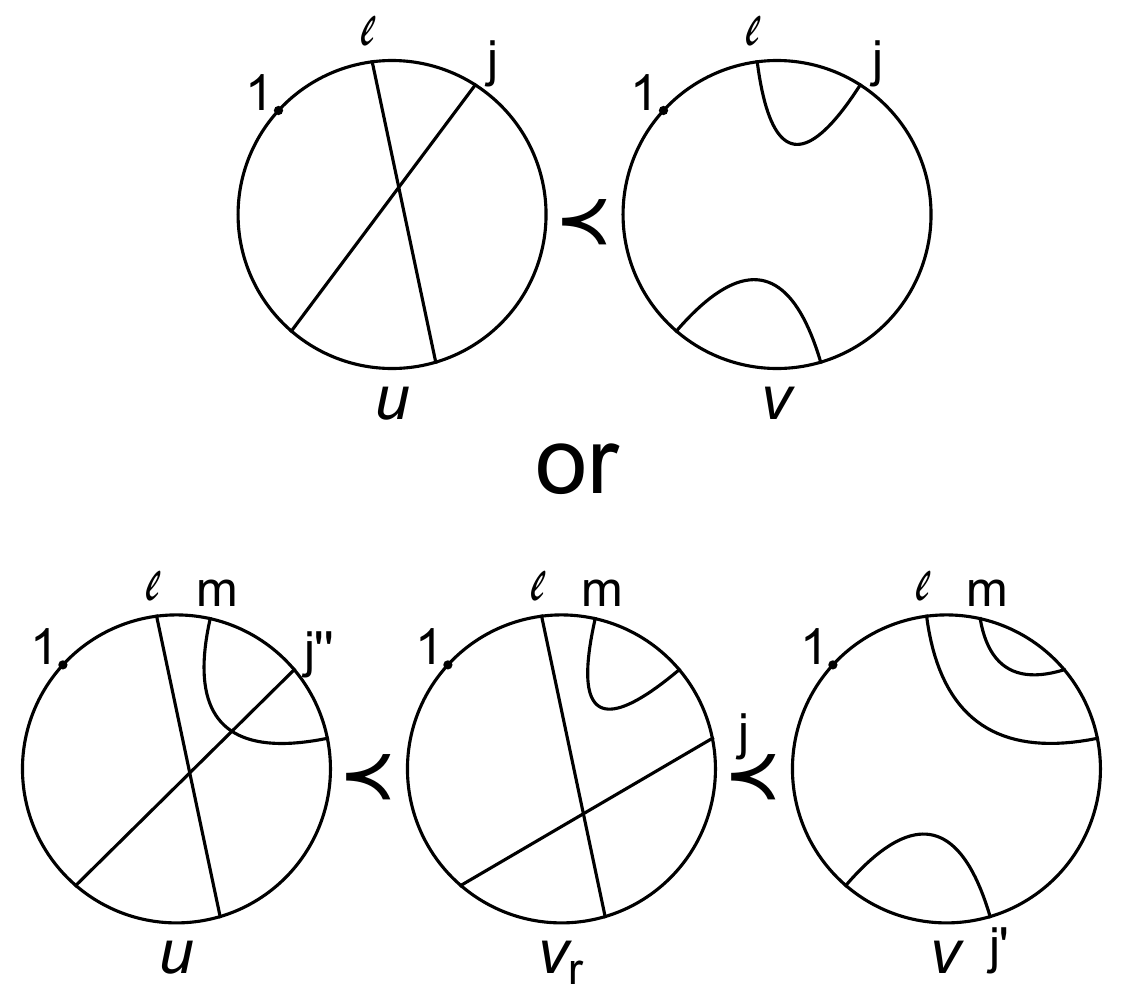}
\caption{$\lambda (v_r,v) = (j,l)$}  
\end{center}
\end{figure}


Next observe  that the value 
$j$ in the label $\lambda (v_r,v) = (j,l)$ having $j>l$ with $v_r\prec v$ in a topologically ascending chain
is uniquely determined by $w(v)$ and $l$, as follows.
The position of  $j_p$ in $w(v_r)$ is  the position of $l_v$ in $w(v)$,  allowing us to 
determine $j$ from $w(v)$ and $l$  by virtue of  $w(v_r)$ necessarily 
coinciding with $w(v)$ to the left of this position, as explained in Remark ~\ref{two-types-uncrossings}.

  Now suppose there are  two distinct  cover relations   $v_r\prec v$ and $v'_r \prec v$ 
 downward from $v$ both having the same label $(j,l)$ for $j>l$; moreover, suppose that
  $v_r $ and $v'_r$ both belong to topologically 
 ascending chains from $u$ to $v$.    Let us  first check  that
  this necessarily implies that 
 $w(v)$ has a subsequence (a)  $l,l,j,j, t,t$ or (b) $l,l,j,t,j,t $   for $l<j<t$;   the
  other possibility, namely having  the subsequence 
  $l,l,j,t,t,j$ appearing in $w(v)$,  is ruled out 
  by virtue of the fact that  a cover relation must eliminate a single crossing.  See Figure 10.
  Specifically, the need 
  for cover relations     precludes  
  nesting  between the $j$ and $t $ wires in $v$, since the existence of distinct $v_r$  and 
    $v'_r$   necessarily means that  among the downward cover relations $v_r\prec v$ and 
    $v'_r\prec v$, one of these must cross the $l$ and $j$ wires from $v$ while the other must cross
    the $l$ and $t$ wires from $v$.  One thing that may be confusing here is that both cover relations do   receive the same  label $(l,j)$ in spite of one of them involving the $l$ and $t$ wires; this 
    is because the names for the wires, for purpose of labeling  a cover relation, 
    are determined at  the lower element of the cover relation.
  Regardless of whether we are in case (a) or (b), let us make the convention that  $v_r$ is obtained from $v$ by crossing the $l$ and $j$ wires from $v$, while $v'_r$ is obtained from $v$ by crossing the $l$ and $t$ wires from $v$.

\begin{figure}[htb]
\begin{center}
\includegraphics[height=2in,width=5in,angle=0]{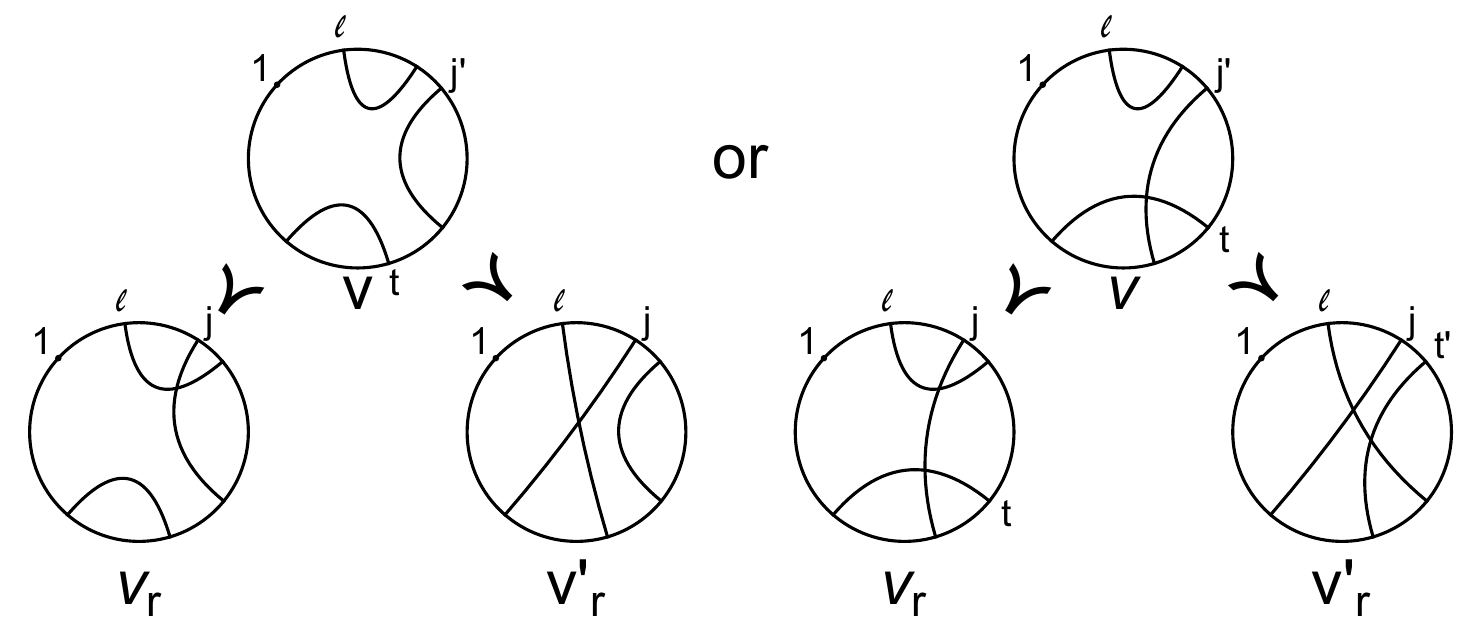}
\caption{$v_r\prec v$ and $v_r'\prec v$}  
\end{center}
\end{figure}

 Now we turn to the task of ruling out  (a), namely the case  where 
  $v_r $ replaces subsequence  $l,l,j,j,t,t$ in $w(v)$ having $l<j<t$ with subsequence  $l,j,l,j,t,t$  in $w(v_r)$
 while $v'_r$ instead replaces $l,l,j,j,t,t$  with subsequence  $l,j,t,t,l,j$ in $w(v'_r)$. 
 We will use the fact  that  saturated chains downward from $v_r$  to $u$ and from $v'_r$ to $u$ eventually do reach a common  element below both of them, so in particular a single shared
  start set  at  this common element; each topologically ascending  
 saturated chain which includes  $v_r$  will therefore  need a label  $(t,l)$ 
  somewhere lower in the saturated chain so as to move the leftmost copy of $t$ leftward to its 
  position in this common start set.  To see why we definitely will need such a label $(t,l)$, it helps to 
  notice that the part of the chain which impacts the start set is limited to downward steps which each move a single label of the form $i_v $ to the right, moving a label of the form $j_p$ into the position $i_v$ had occupied; although  it is possible that the positions of $l_p,j_p,t_p$ could move even farther to the left prior to reaching a common lower bound for $v_r$ and $v'_r$,  that would necessitate a larger label than $(j,l)$ lower in each saturated chain, forcing descents (and hence topological descents by Lemma ~\ref{descent-is-topol-descent}) in the aforementioned  proposed  topologically ascending saturated chains containing $v_r$ and $v_{r'}$, 
  enabling us to rule out that possibility.   Thus, we deduce 
  the clam about needing the label $(t,l)$ lower in our saturated chain downward from $v_r$ to a common lower bound.

  This lower copy of the label $(t,l)$ in the proposed saturated chain involving $v_r$ 
   will force a topological descent somewhere in the saturated chain,  as we explain
 next.  Proposition ~\ref{j-i-topol-descent} directly handles the possibility of consecutive labels 
 $\lambda (v_{r-1},v_r) = (t,l)$ and $\lambda (v_r,v) = (j,l)$ of the form described  
 above with $t$ chosen as small as possible, by the case analysis in the proof of Proposition ~\ref{j-i-topol-descent} yielding a topological descent in the case that describes our scenario 
 (which translates to case (c) in the proof of Proposition ~\ref{j-i-topol-descent}).  In  the ``non-consecutive case'',  namely the case where $(t,l)$ appears lower in the saturated chain rather than directly below the label $(j,l)$ with the further assumption that the intermediate labels are not
 all of the form $(t',l)$ for $j<t'<t$,  
 we use   the fact that there will be one or more  other labels at intermediate positions.  This would necessarily force a descent (and hence a topological descent by Lemma ~\ref{descent-is-topol-descent}) somewhere on the segment of labels beginning and ending with these two labels, by virtue of some  label at an intermediate position necessarily either being smaller than both of these labels $(t,l)$ and $(j,l)$ or being  larger than both of these labels, due to our very assumption about the labels with second coordinate $l$ and larger first coordinate being non-consecutive.   The upshot is that  
 we get a contradiction to having $v_r \prec v$ and $v'_r\prec v$ both labeled $(j,l)$ and both belonging to topologically ascending chains from $u$ to $v$ when we are in  case (a) above, namely  the case
 with $w(v)$ including subsequence $l,l,j,j,t,t$.  

 Case (b), namely the case  with subsequence $l,l,j,t,j,t$ in $w(v)$ again having $l<j<t$, is  
 likewise  ruled out by  a completely analogous argument which  will be 
 largely left to the reader.  What  makes the argument  work again  in this case 
 is that  $w(v_r) $ now has subsequence $l,j,l,t,j,t$ and $w(v_{r-1})$ has subsequence 
 $l,j,t,l,j,t$ in the event of consecutive labels $\lambda (v_r,v) = (j,l)$ and 
 $\lambda (v_{r-1},v_r) = (t,l)$, which means that when we now apply  
 Proposition ~\ref{j-i-topol-descent} in this case, we again find ourselves in a scenario giving a descent (and hence a 
 topological descent by Lemma ~\ref{descent-is-topol-descent}), yielding a contradiction; that is, we find ourselves in  the 
 scenario labeled as case (b) within the proof of  Proposition ~\ref{j-i-topol-descent}. 
\end{proof}

Next we handle a situation that required special care within the proof of Lemma ~\ref{semi-mod}, 
hence was split off as a separate proposition.

\begin{proposition}\label{j-i-topol-descent}
Given $u\prec v \prec w$ in $P_n^*$ with $\lambda (u,v) = (k,i)$ and $\lambda (v,w) = (j,i)$ for $i<j<k$ carrying out a pair of consecutive uncrossing steps involving a total of three wires, then there exists $u\prec v' \prec w$ either with label sequence $\lambda (u,v') = (j,i)$ and 
$\lambda 
(v',w) = (j',k')$ for some $j'<k'$  or with label sequence $\lambda (u,v') = (j,k)$ and $\lambda (v',w) = (j,i)$.  In the former case, $u\prec v\prec w$ comprises a topological ascent, and in the latter case $u\prec v \prec w$ comprises a topological descent.
\end{proposition}

\begin{proof}
The existence of $u\prec v$ with $\lambda (u,v) = (k,i)$ implies that $w(u)$ has subsequence
$i,k,i,k$.   The $i<j<k$ requirements implies that the first copy of $i$ is to the left of the first copy of
$j$ which is to the left of the first copy of $k$.  These restrictions imply that 
the only viable  possibilities for the subsequence of $w(u)$ with letters $i,j,k$ are 
(a) $i,j,k,i,k,j$, (b) $i,j,k,i,j,k$, (c)  $i,j,k,j,i,k$ or (d) $i,j,j,k,i,k$.  See Figures 11, 12, 13 and 14 for illustrations of how we handle these cases.
   In each case, we will  use the result of Thomas Lam from \cite{La14a} 
that $P_n$ (and hence $P^*_n$) is  Eulerian; it follows immediately from this 
and the gradedness of $P^*_n$ that $\mu_{P_n^*} (u,w) = (-1)^2$.  This 
in turn implies for each saturated chain 
$u\prec v\prec w$ the existence of a unique element $v'$ satisfying $u\prec v' \prec w$, by 
Remark ~\ref{eulerian-then-thin}.

\begin{figure}[htb]
\begin{center}
\includegraphics[width=3in,angle=0]{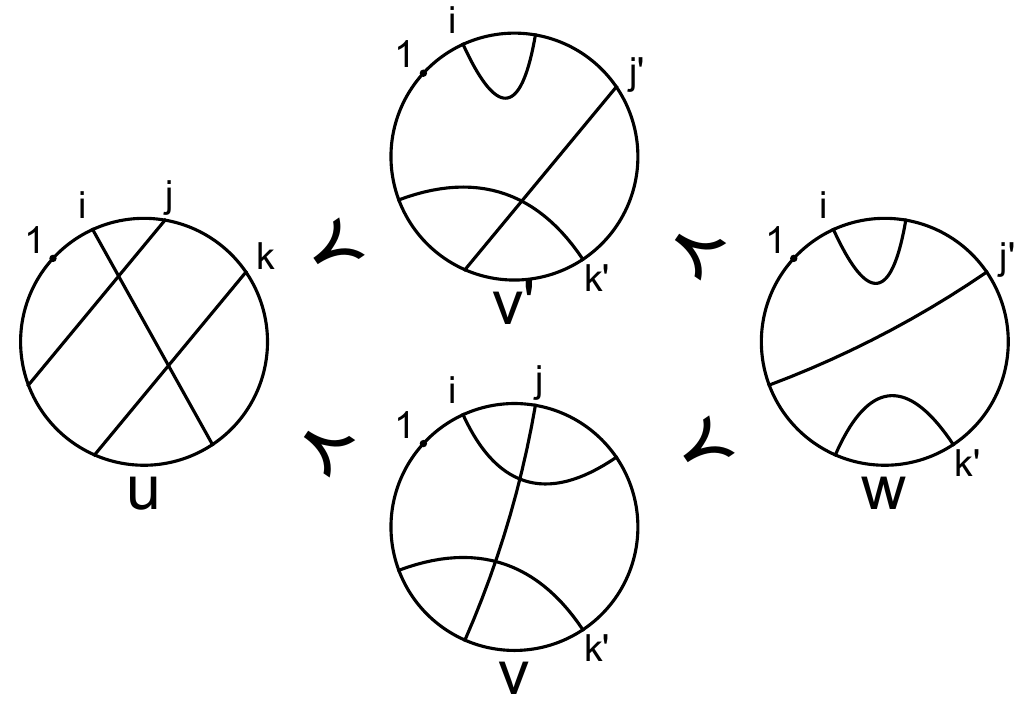}
\caption{Case a: $u\prec v \prec w$ topological ascent}  
\end{center}
\end{figure}

\begin{figure}[htb]
\begin{center}
\includegraphics[width=3in,angle=0]{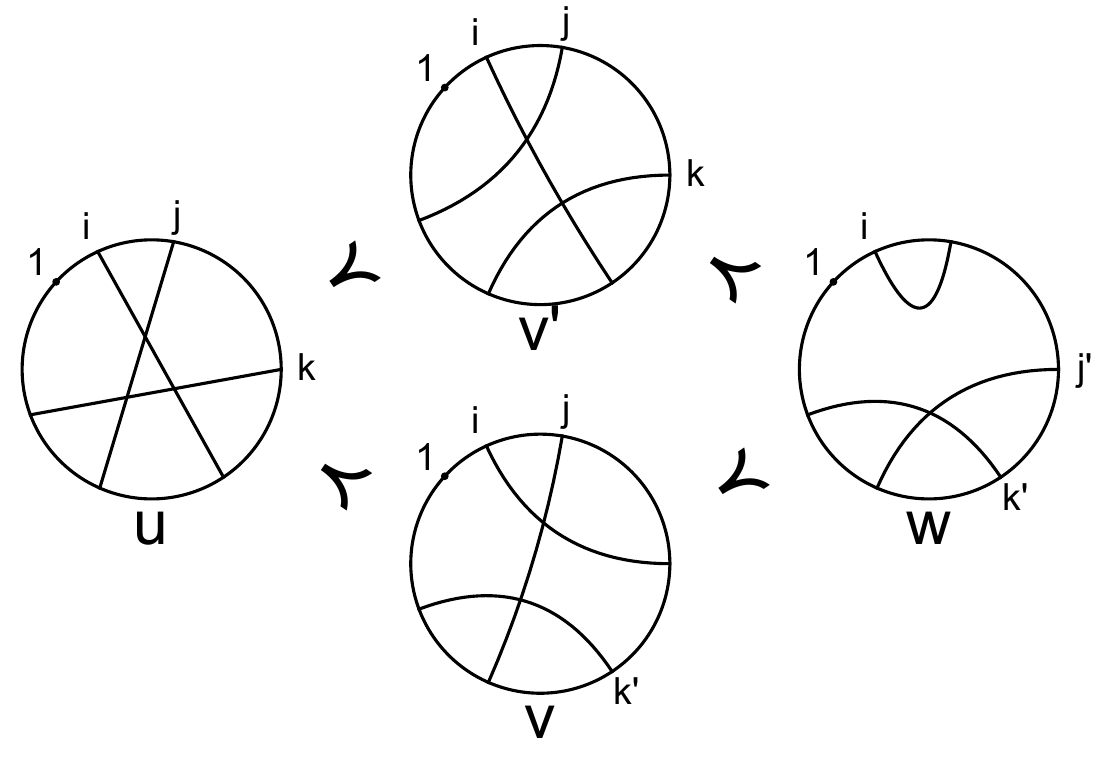}
\caption{Case b: $u\prec v \prec w$ topological descent}  
\end{center}
\end{figure}

\begin{figure}[htb]
\begin{center}
\includegraphics[width=3in,angle=0]{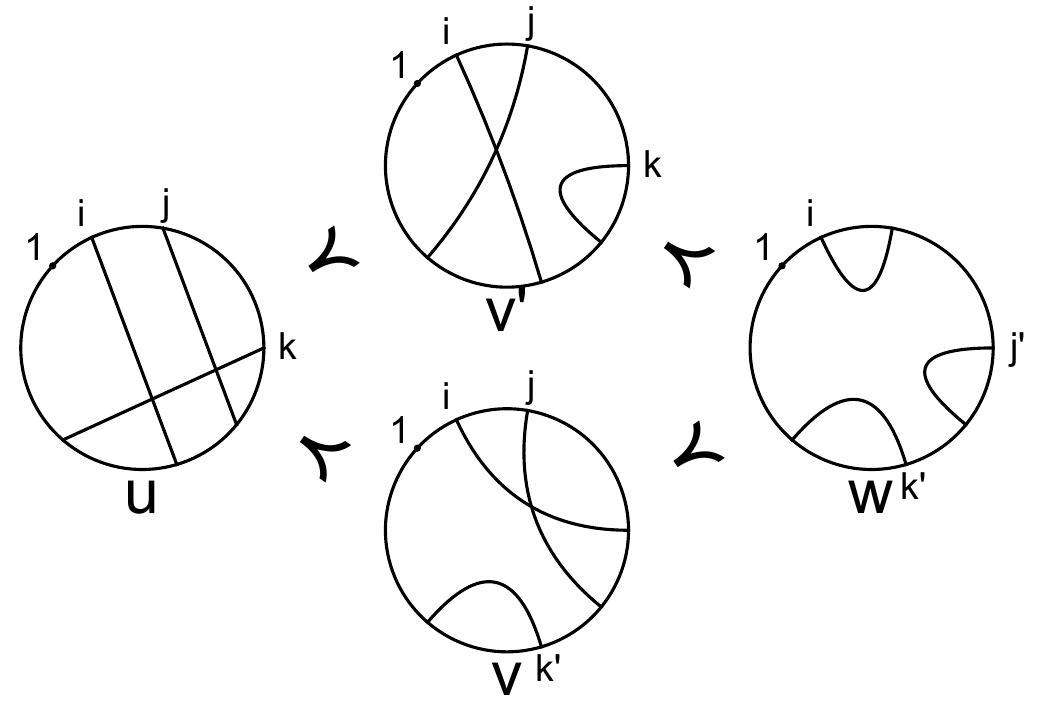}
\caption{Case c: $u\prec v\prec w$ topological descent}  
\end{center}
\end{figure}

\begin{figure}[htb]
\begin{center}
\includegraphics[width=3.5in,angle=0]{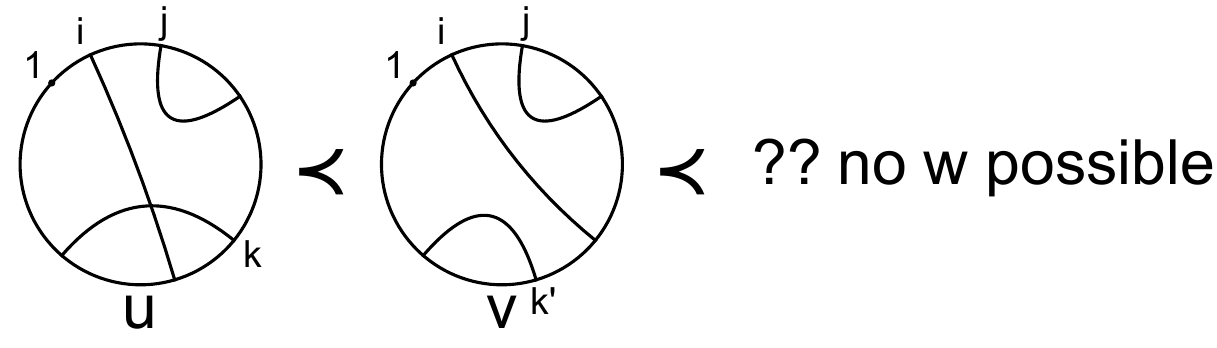}
\caption{Case d: cannot arise due to contradiction}
\end{center}
\end{figure}

In case (a), namely the case with $w(u)$ having the
subword $i,j,k, i, k,j $, the $i$ wire crosses both the $j$ wire and the $k$ wire in $u$, but there is no crossing of the $j$ and $k$ wires in $u$. 
The three wires $i,j,k$ have a total of two crossings, which may be uncrossed in either order to obtain 
$w$.   Observe that one of the uncrossing sequences yields the labels $\lambda (u,v) = (k,i) $ and $\lambda (v,w) = (j,i)$, while   the other uncrossing sequence yields  $\lambda (u,v') = (j,i)$ and $\lambda (v',w) = (j',k')$ for some $j'<k'$.   These are the only saturated chains from 
$u$ to $v$ since both these uncrossings must be accomplished.  The latter saturated chain from $u$ to $v$ gives a descent and indeed is the lexicographically later of the two sequences, hence a topological descent.    Thus, $\lambda (u,v) = (k.i) $ and $\lambda (v,w) = (j,i)$ together give a topological ascent in this case.

For the cases (b) and (c), namely the cases  with $w(u)$ having  subsequences
$i,j,k,i,j,k$ and $i,j,k,j,i,k$, respectively,  a similar analysis applies, yielding that $u\prec v \prec w$ is a topological descent in these cases.
The $k$ wire crosses  both the $i$ wire and the $j$ wire in $u$ in each of  these cases.  
One may check both for case  (b) and for case  (c) 
that the unique  $v'$ satisfying $u\prec v' \prec w $ 
yields $\lambda (u,v') = (j,k)$ with $j<k$.  
Our order  $<_{\lambda }$  implies
$(j,k) <_{\lambda } (k,i)$ since the former has $j<k$ while the latter has $k>i$.  One may also
observe that in each case we have 
$\lambda (v',w) = (j,i)$, yielding the desired lexicographically earlier 
$u\prec v' \prec w$.  In particular, in each of the cases (b) and (c), we see that
 $u\prec v\prec w$ is a topological descent, as desired.  

In case (d), the case with $w(u)$ containing the subsequence $i,j,j,k,i,k$,  we deduce from   $\lambda (u,v) = (k,i)$ that  $w(v) = i,j,j,i,k,k$.  This contradicts the existence of a cover relation $v\prec w$ uncrossing the $i$ and $j$ wires, since these wires are nested rather than crossing in $v$.  Thus, (d) is ruled out.  
\end{proof}

  \begin{remark}
  {\rm 
  Figure 4 also exhibits the fact that the edge labeling $\lambda $ is not an EL-labeling in general, 
  because there are rank 2 intervals having two different ascending chains, the lexicographically later of which is a topological descent but not an actual descent.  Such examples are what led us instead to  prove
   that $\lambda $ satisfies the more relaxed requirements to be an EC-labeling, which still yields a lexicographic shelling.
   }
  \end{remark}
  
Now we turn to intervals $[D,\hat{1}]$.  We begin by showing that the lexicographically first saturated chain from $D$ to $\hat{1}$ (with respect to edge labeling $\lambda $)  has weakly ascending labels.  This will be a key tool to proving in Lemma ~\ref{0-unique} that this saturated chain  is the only topologically ascending saturated chain from $D$ to $\hat{1}$.

\begin{lemma}\label{0-exists-first}
For each $D< \hat{1}$ in $P_n^*$, 
the lexicographically first saturated chain from $D$ to $\hat{1}$ (with respect to edge labeling $\lambda $)  has weakly ascending labels.
\end{lemma}

\begin{proof}
We may assume $D$ has at least one crossing, since otherwise $D$ is covered by $\hat{1}$,  a vacuous case.  
Now let us show how to construct a saturated chain from $D$ to $\hat{1}$
with weakly ascending labels.
We start by greedily choosing the smallest 
possible $i$ for which there exists at least one wire $k$ that crosses the 
wire $i$ for $k>i$.   Among such wires that cross  wire $i$, choose the smallest 
$j$ such that wires $i$ and $j$ cross.   See Figure 15.
Lemma ~\ref{noncross-down} justifies that we may proceed 
up  a cover relation  $D\prec D'$ in $P_n^*$ 
by uncrossing these wires $i$ and $j$ in such a way  
that $w(D')$ has the subsequence $i,j, j, i$.  This  ensures  
  $\lambda (D,D') = (i,j)$ with $i<j$, 

\begin{figure}[htb]
\begin{center}
\includegraphics[height=1.25in,width=2.5in,angle=0]{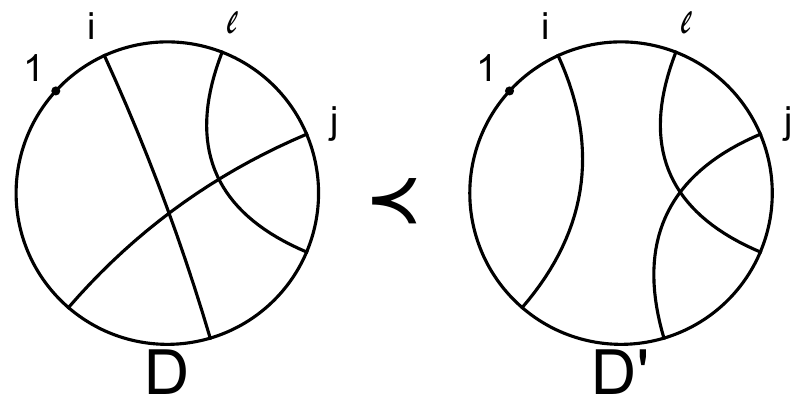}
\caption{$\lambda (D,D') = (i,j)$} 
\end{center}
\end{figure}

Next we  show that the smallest available label on any cover relation 
upward from the diagram  $D'$ obtained this way is no smaller than 
$(i,j)$. 
To this end, we  analyze the impact of the 
exchange  of $i_v$ and $j_v$ together with the corresponding
uncrossing of wires $i$ and $j$; specifically, we need to constrain how this may change 
  the names of  any pairs of wires that still cross each
other in $D'$ from what their names are in $D$.
If such renaming of wires were to cause a smaller label than $(i,j)$  
to be available for a cover relation upward
from $D'$, this new label would necessarily result from the 
renaming of some $(i,k)$ crossing as $(j,k)$ for $j<k$ and the  
renaming of  some $(j,l)$ crossing as $(i,l)$ for $i<l$, by virtue of exchanging portions of wires 
$i$ and $j$.
Only the new potential label $(i,l)$  could be  smaller than $(i,j)$, and this 
would only happen for $i<l<j$.  But the fact that the $i$ and 
$l$ wires do not cross in $D$,  
together with the fact that  the pairs  $(i,j)$ and $(j,l)$ both  do cross in $D$ may all be
combined to deduce that 
$w(D) $ has the subsequence $i,l,j, l,i,j $.
That is,  we have  $l_v$ before  $i_v$ which is before $j_v$ as we proceed clockwise from
starting point $1_p$.  
Exchanging $i_v$ and $j_v$ to obtain $w(D')$ from $w(D)$ will yield $D'$ that 
preserves the fact that $l_v$ comes earlier than $i_v$ in $w(D')$.  This 
contradicts  the availability of the label $(i,l)$ for a cover relation upward from $D'$, since we have
just shown that the  wires $i$ and $l$ do not cross each other  in $D'$.

Applying the above argument  repeatedly as we proceed up a saturated chain greedily choosing
the lexicographically smallest available label at each step, 
we may conclude that each pair of consecutive labels $\lambda (x,y) $ and $\lambda (y,z)$ for
$z < \hat{1}$ is weakly ascending.   By virtue of the construction above, also notice that each
label $\lambda (y,z) = (i,j)$
for $z<\hat{1}$ in the lexicographically first saturated chain has $i<j$, implying the label is smaller
than $L$.  Thus, we also will get an ascent for the pair of labels $\lambda (x,y) $ and 
$\lambda (y,\hat{1})$ at the last step in our lexicographically first saturated chain.  
\end{proof}

Now we complete the case of intervals $[D,\hat{1}]$, showing that  $\lambda $ meets the requirements of an EC-labeling on
all such intervals. 

\begin{lemma}\label{0-unique}
For each $D<\hat{1}$ in $P_n^*$, every saturated chain 
from $D$ to $\hat{1}$  that is not lexicographically first  (with respect to edge labeling $\lambda $) has a topological 
descent.  
\end{lemma}

\begin{proof}
Consider a saturated chain $N = D\prec u_1\prec \cdots \prec u \prec v \prec \cdots \prec \hat{1}$ from 
$D$ to $\hat{1}$  
such that there is $u\prec v$ in $N$  
with $\lambda (u,v') < \lambda (u,v)$ for some $u\prec v' < \hat{1}$. This implies that
there is a lexicographically earlier 
saturated chain from $u$ to $\hat{1}$ involving $v'$ instead of $v$.
 By induction on  $rk(\hat{1}) - rk(v)$, we may assume that the restriction $M$ of  $N$ to the 
 interval $[v,\hat{1}]$
 is the lexicographically earliest saturated chain 
 from $v$ to $\hat{1}$.  The labeling of $M$ must  then consist entirely of 
labels $(i,j)$ having $i<j$ prior to our final label $L$,  by 
Lemma ~\ref{0-big-small} and
Lemma ~\ref{0-exists-first}. 

In particular, $N$  has a descent 
at $v$ unless $\lambda (u,v)$  
is  of the form $(r,s)$ for some $r<s$.  
When there is  some $z \in P_n^* \setminus \{ \hat{1} \} $ 
that is greater than both $v$ and 
$v'$,  we may deduce the desired result from
Lemmas ~\ref{semi-mod} and ~\ref{descent-is-topol-descent}.
We confirm shortly  that  we  will indeed  have such 
    an upper bound $z$  unless $v$ and $v'$ are obtained from $u$ 
  by uncrossing the same pair of wires in the two different possible ways; 
   but this
  uncrossing of   the same pair of wires $p$ and $q$ 
  in opposite ways     would  give labels  
  $\lambda (u,v') = (p,q)$ and $\lambda (u,v) = (q,p)$ for  $p<q$, contradicting the fact that
  $\lambda (u,v) = (r,s)$ for some $r<s$ with $\lambda (u,v') < \lambda (u,v)$.
     
 Now to the outstanding claim.   In the event that we do not uncross the same pair of wires in
 different ways to obtain $v$ and $v'$, we either move from $u$ to $v$ and $u$ to $v'$ by uncrossing disjoint pairs of wires, or via crossings that share one wire in common.  In either of these  cases, there exists a wire diagram that is an upper bound for $v$ and $v'$
 covering $v$ and $v'$ by doing both of these uncrossings and no other uncrossings.  
\end{proof}

Finally, we give a simple technical tool that was used in the proof of Lemma ~\ref{0-unique}.

\begin{lemma}\label{0-big-small}
For each $D<\hat{1}$ in $P_n^*$, every  saturated chain from $D$ to $\hat{1}$ that uses any labels $\lambda (u,v)$ 
of the form  $\lambda (u,v) = (j,i)$ for $j>i$  has a   
descent.  
\end{lemma}

\begin{proof}
Note that any cover relation label  $\lambda (x,y) = (j,i)$ for $j>i$  
is larger than $L$, while every saturated  chain upward from $D$ to $\hat{1}$  has $L$ as its final label.  
This already guarantees the presence of a descent  on any  saturated chain from $D$ to $\hat{1}$ involving a label $(j,i)$ for $j>i$.
\end{proof}

\section{Shelling all intervals in Tuffley posets}\label{Tuffley-section}

First we discuss in general terms  the edge product space of phylogenetic trees and then we define more precisely  its face poset, the \emph{Tuffley poset}. 
These are both discussed in much more detail in for instance \cite{MS} and \cite{GLMS}; 
\cite{MS} gives a CW decomposition for this space while \cite{GLMS}  proves that  this is a regular CW decomposition by a proof that involves  first proving the existence of a dual EC-shelling 
for each interval in the Tuffley poset.    We will give a much more explicit shelling for each interval in the Tuffley poset, by showing that 
each interval is isomorphic to an interval in the poset $P_n$ (which we have already proven to be shellable in the earlier sections of this paper).

This will require the notion of the minors of a graph (with sets of labels on some of its vertices), 
as well as notions of edge contraction, edge deletion, and safe edge deletion in this context.

  \begin{definition}
   
  An {\bf edge contraction} in a graph with sets of labels on some of its vertices shortens an edge in the graph  to length 0, 
  identifying its two endpoints with each other and merging their associated (possibly empty) sets of labels to give the set of labels 
  assigned to the new merged vertex.

  An {\bf edge deletion} simply eliminates an edge while keeping its two endpoints as distinct vertices and preserving their associated label sets.
    An edge deletion is said to be {\bf safe} if it does not result in any vertices with empty label set dropping to degree exactly 2.  
  
   A {\bf minor} of a graph $G$ is any graph obtained from $G$ by a series of edge contractions followed by a series of safe edge deletions.  A minor of one of our upcoming labelled trees will be enriched with a vertex labeling induced by  the labeling of the leaves in the tree for which it is a minor.

  Two minors are said to be {\bf combinatorially equivalent} if there is a graph isomorphism from one to the other that maps the vertex labeling of one minor to the vertex labeling of the other minor.
  \end{definition}

It may be helpful to note that the above  notion of combinatorial equivalence classes  for trees with labels on all of the leaves  
 may alternatively  be defined using  the  set of  ``splits'' of a tree.  
  A {\bf split}  of a tree $T$ with $n$ leaves labeled $1,2,\dots ,n$ 
is a set partition of $\{ 1,2,\dots ,n\} $ into two blocks that is  obtained by deleting one   edge from $T$ and letting the  blocks of the set partition be the sets of leaf labels for  leaves  in the same connected component as each other in the resulting  forest (comprised of  two trees).  A pair of  trees with labeled leaves are 
 combinatorially equivalent  if they have the same splits. 

\begin{definition}
The {\bf edge product space $\varepsilon (X) $ of phylogenetic trees} with leaf label set $X = \{ 1,2,\dots ,n\} $ is a stratified space comprised of cells.   The  maximal open cells  of 
$\varepsilon (X)$ are indexed by the combinatorial 
equivalence classes of  trees $T$ with $n$ leaves (i.e. $n$ nodes of degree 1) such that each leaf is assigned a
distinct label from $X$ 
 and each non-leaf  node in $T$ has degree exactly 3.
 One may check that the trees $T$ indexing the maximal cells all have exactly $2n-3$ edges, i.e.  have $|E(T)|=2n-3$.

The open cell $C(T) $ given by tree  equivalence class $T$  with  $|E(T)|$ edges in $T$  may be parametrized as  the points in 
$\RR_{>0}^{|E(T)|}$, with the $|E(T)|$  coordinates  recording the lengths of the corresponding edges
in $T$, making it an open cell of dimension $|E(T)|  = 2n-3$.  
The lower dimensional cells $C(T')$ in the closure of $C(T)$ are given by exactly those combinatorial equivalence classes of labeled forests $T'$ obtained from $T$ as minors of $T$.


The space $\varepsilon (X)$ also has  an open cell homeomorphic to $\RR_{>0}^{|E(T')|}$  for  each (combinatorial equivalence class of) labeled forest $T'$  
obtainable  by choosing some $T$  as above indexing a maximal cell  $C(T)$ 
and  letting $T'$ be one of its minors obtained by  letting one or more  of the 
  edge lengths in $T$  degenerate either to 0  (via an edge contraction) or  to infinity (via a safe edge deletion).
  \end{definition}
   
Next   we define the face poset for the edge product space of phylogenetic trees.  This is denoted by $S(\{ 1,2,\dots ,n\} )$ in \cite{GLMS}, but we instead denote it by $T(n)$ to avoid confusion with our notation for  ``start set'' earlier in the paper.   

  \begin{definition}\label{Tuffley-def}
The {\bf Tuffley poset} $T(n)$ is the face poset for the edge product space of phylogenetic trees with $n$ leaves.  
It has 
as its maximal elements  the various combinatorial equivalence classes of trees $T$  with $n$ leaves labeled 1 to $n$ with exactly one label on each leaf with  the further requirement that  each non-leaf  node has degree exactly  3.   The non-maximal elements of $T(n)$  other than $\hat{0}$ are those
forests $F$  (with collections of  labels on some of the vertices) which  may be obtained from a maximal element of $T(n)$ 
by repeatedly proceeding down cover relations  as follows.      Given any $v\in T(n)$,  first do a  (possibly empty) series of edge contractions, then do a (possibly empty) series of safe edge deletions.
Elements $u$ obtained in this manner which are combinatorially equivalent to each other
give the same poset element.   
 
  This process  terminates at  (combinatorial equivalence classes of)  graphs which do not have any edges and which have nonempty collections of labels assigned to each surviving  vertex; 
these 
are in natural bijection with the set partitions of $\{ 1,2,\dots ,n\} $.  
Finally, a unique minimal element in $T(n)$, denoted
$\hat{0}$, is adjoined. 
\end{definition}  

Now we are equipped  to deduce the  shellability of each interval in the Tuffley 
poset $T(n)$  as a corollary to Theorem ~\ref{EL-shell-proof} by verifying the requisite poset isomorphisms.    

\begin{example}\label{small-isom}
Small examples that could  be helpful to  keep in mind while reading the proof of Corollary ~\ref{Tuffley-corollary} are 
the Tuffley poset $T(2)$  for a tree with two leaves and no non-leaf vertices and the Tuffley poset $T(3)$ for a  tree with three leaves and one degree 3 non-leaf vertex.  The former tree  will be the medial graph  in the proof below for a wire diagram with two wires that cross each other, leading to each interval in $T(2)$ being isomorphic to an interval in the uncrossing poset $P_2$.  The latter tree will be the medial graph for  a wire diagram with three wires all crossing each other, leading to each interval in $T(3)$ being  isomorphic to an interval in $P_3$.  
\end{example}

\begin{corollary}\label{Tuffley-corollary}
Each interval in $T(n)$ 
is dual EC-shellable by way of a poset isomorphism  to an interval in  the uncrossing poset $P_n$.
 This  shelling may be combined with other known properties to  imply  that the CW decomposition given by 
 Moulton and Steele in  \cite{MS} for the edge product space of phylogenetic trees is a regular CW decomposition.
\end{corollary}

\begin{proof}
Our main task will be to give an isomorphism of each interval of the Tuffley poset  $T(n)$ to an interval in an  
uncrossing poset $P_n$.   Then we may use the fact that our EC-shelling on 
the dual to the  uncrossing poset induces an EC-shelling on each of its intervals.    The fact that an EC-shelling 
of  a poset  induces an EC-shelling on each of its intervals is  immediate 
from the definition of EC-labeling.

  \begin{figure}[h]\label{well-connected-figure}
        \begin{picture}(325,90)(-40,-20)

  \put(-63,1){\line(1,0){30}}
 \put(-63,1){\circle*{4}}
 \put(-33,1){\circle*{4}}      
  
  \put(-52,-17){$G_2$}
 

  \put(-3,1){\line(1,0){60}}
  \put(27,1){\line(0,1){30}}  
    \put(-3,1){\circle*{4}}
 \put(57,1){\circle*{4}}      
 \put(27,31){\circle*{4}}
         
\put(23,-17){$G_3$}

       \put(77,1){\line(1,0){90}}
 \put(77,1){\circle*{4}}
 \put(167,1){\circle*{4}}     
 \put(107,1){\line(0,1){30}}
 \put(137,1){\line(0,1){30}}
 \put(107,31){\circle*{4}}
 \put(137,31){\circle*{4}}
 \put(107,31){\line(1,0){30}}

 \put(118,-17){$G_4$}
 
 
       \put(197,1){\line(1,0){120}}
 \put(197,1){\circle*{4}}
 \put(317,1){\circle*{4}}     
 \put(227,1){\line(0,1){30}}
 \put(287,1){\line(0,1){30}}
 \put(227,31){\circle*{4}}
 \put(287,31){\circle*{4}}
 \put(227,31){\line(1,0){60}}
 
\put(257,1){\line(0,1){60}}
\put(257,61){\circle*{4}}

\put(253,-17){$G_5$}
                
        \end{picture}
        \caption{\label{16}The well connected graphs $G_2$, $G_3$, $G_4$ and $G_5$} 
\end{figure}


In constructing a poset isomorphism, we will rely heavily  on the notion of the medial graph associated to a wire diagram (see e.g. Section 4.4.1 in \cite{Ke} and Figure 3 in \cite{Ke} for a specific example of this construction, in which context the term strand diagram is used for what we are calling a wire diagram).   Roughly the idea for this association  is to have each crossing of two wires in a wire diagram  go through the middle of an edge of the associated medial graph, with  the edges of the medial graph  in bijection with the wire crossings in the associated wire diagram.   We will also use the well known fact that the medial graph associated to one choice of embedding of a  wire diagram having all $n$ wires  cross each other, i.e. a  wire diagram  associated to  the maximal element in  the dual uncrossing poset $P_n^*$,   is a so-called well-connected graph $G_n$.  Changing the embedding of the wire diagram  by passing a wire over a crossing (a sort of analogue of a braid move)  will correspond to a so-called Y-$\Delta $ move in the associated graph (see \cite{Ke}) which produces a different well-connected graph.  

Results in Section 3 of \cite{KW}  imply 
that each planar graph (and in particular each tree equivalence graph $T$ with unlabelled leaves) arises as a minor of a well
connected graph, namely as an element $v\ne \hat{1}$ in some dual uncrossing poset $P_m^*$ for $m\ge n$.  
It is not difficult to prove this directly in our setting where we restrict to trees with $n$ leaves such that each non-leaf node has degree 3, by  embedding any such  tree  into the well-connected graph $G_n$ (see Figure \ref{16}).   In case it may help the reader, this is carried out very explicitly in Lemma ~\ref{embedding-trees-in-G_n}.    
Next observe that embellishing such a  tree $T$ by assigning labels to its leaves does not interfere with the map from \cite{KW} mentioned above  identifying each  $T$  with an element $v$ in $P_m^*$ (or alternatively  with the explicit construction of such a map given in the proof of  Lemma ~\ref{embedding-trees-in-G_n}).  It is also straightforward to check  that  this embellishment of $T$ with  leaf  labels will not interfere with the poset interval isomorphism, as  we describe and justify  next; here it will be important that the boundary nodes in $G_n$ regarded as an electrical network  are exactly those along the upper boundary of $G_n$ (so that we do not introduce any additional safe edge deletions beyond those appropriate to the Tuffley  poset).  

Equipped with this,  let us now show that the interval $[\hat{0},v]$ in $P_m$ is naturally isomorphic to the interval $[\hat{0},T]$ in the Tuffley poset, focusing on the situation without leaf labels and then leaving it to the reader to think through the fact that leaf labels behave just as they should, using our labeled version of edge deletion and contraction and only allowing safe edge deletions.  
The point is to convert the strand diagram for $v\in P_m$ to a minor $G_v$ of a well-connected graph using the bijection appearing in Section 4.4.1 (sending a wire diagram to its medial graph)  in 
\cite{Ke} to guide us.  Observe that the two different ways of uncrossing a pairs of wires exactly correspond to deletion (for one type of uncrossing) and contraction (for the other way of uncrossing the crossing)  of the graph edge corresponding to that wire crossing in the associated medial graph.  Such an uncrossing yields a double crossing if and only if the corresponding edge deletion or contraction yields what is called in \cite{Ke} a non-reduced graph; we will show in our setting (where we start with a tree with leaf labels) that edge contractions never yield non-reduced graphs and that edge deletions yield non-reduced graphs if and only if they are unsafe edge deletions.   To this end, we show next  that because we start with a tree, the only way we can get a non-reduced graph as a minor is to have the edge deletion or contraction lead to the existence of an internal vertex  (i.e. a  vertex not having any leaf labels on it)  of degree 1 or 2.  This situation corresponds to exactly the edge deletions which are not safe (namely exactly those which are not allowed for cover relations in the Tuffley poset) 
and cannot happen at all  for edge contractions for graphs that are trees or forests.  This will suffice since all minors of trees are themselves either trees or forests.  

To  see this claim, notice that a double crossing of wires means that the same two wires cross in the middle of one graph edge  $e_1$ and also in the middle of another edge $e_2$,  which can only happen  either (a) as a result of either a degree 2 vertex that is not a boundary vertex of the associated planar electrical network (in which case the two wire crossings are in the middle of  the two edges $e_1$ and $e_2$  incident to this vertex) or (b)  from a series  of  connected edges comprising a graph path that begins with $e_1$ and ends with $e_2$  (such that one wire crosses each of the edges in this path)   together with  a distinct series of graph edges (comprising a different path starting with $e_1$ and ending with $e_2$ (such that each edge in this path is  crossed by  the other wire); in particular, this pair of graph paths in (b)  meeting at both ends implies there is a cycle in the graph, a contradiction to the medial graph being a tree or a forest  as is the case for elements in the Tuffley poset.   Situation (a) is exactly the situation created by an unsafe edge deletion.  
 Thus, we have shown that the edge contractions and the safe edge deletions 
correspond not only to the cover relations in the Tuffley poset but also correspond to the cover relations  in the uncrossing poset.  Putting this together, we have shown that  this map from elements of the interval $[\hat{0},v] $ in $P_m $ to corresponding graph minors of the medial graph  $G_v$ for $v$  is  the desired poset interval isomorphism.  

It is proven in Section 4 of \cite{GLMS}  that the existence of a shelling for each interval of the Tuffley poset implies that the CW decomposition of Moulton and Steel for the edge product space of phylogenetic trees is a regular CW decomposition.  Thus, our  shelling for each interval of $P(n)^*$ (and hence for each interval in $P(n)$ since a poset and its dual have the same order complex)  in Theorem ~\ref{EL-shell-proof} yields regularity of the CW decomposition of Moulton and Steel in an explicit way; by contrast, the earlier  result in \cite{GLMS}  that a shelling for each poset interval exists was non-constructive.
\end{proof}

\begin{lemma}\label{embedding-trees-in-G_n}
Any tree $T$ having  $n$ leaves and having the further property  that each non-leaf vertex has degree 3 may be obtained 
from $G_n$ by a series of edge contractions followed by a series of safe edge deletions.
\end{lemma}

\begin{proof}
To prove this, our main task will be showing how to construct a spanning tree in $G_n$ from which $T$ may be obtained by a series of edge contractions, specifically only using edge contractions for edges which have  the property that one of the two endpoints of the edge has degree 2 just prior to the contraction.   We do this by producing a suitable embedding $\phi $ of $T$ into $G_n$ which maps  the vertices of $T$ to vertices of $G_n$ and maps each edge $e_{v,w}$ of $T$ to a path in $G_n$ from $\phi (v)$ to $\phi (w)$.  Moreover, the $n$ leaves of $T$ will be  mapped exactly to the $n$ vertices (not all of which have degree 1)  running along the upper boundary of $G_n$ (namely the vertices which are solidly shaded in Figure ~\ref{16}).  
Once this has been done, we will be able to obtain $T$ from $G_n$ by first performing precisely this set of edge contractions needed to reduce the spanning tree to $T$, then deleting the edges of $G_n$ what were not used in any of the paths comprising this spanning tree.  One may observe that these edge deletions will all be safe by virtue of each  leaf in $T$ having a label and each non-leaf vertex of $T$ being incident to three edges in $T$; in particular, the endpoints of edges to be deleted must all either have degree at least 3 after the deletion or have a leaf label (due to these upper  boundary nodes of $G_n$ giving exactly  the boundary nodes in the planar electrical network given by $G_n$, 
hence having labels), rendering the edge safe for deletion.  

Now to the spanning tree construction.  We will give a construction that is inductive on the number $n$ of leaves   in $T$. 
Begin by choosing any leaf $l_1$ in $T$.  Let   $\phi (l_1)$ be  the leftmost vertex  
 in $G_n$, denoted $v_{leftmost}$.  Likewise denote the rightmost node in $G_n$ by $v_{rightmost}$.
Next  we map the  unique edge
$e_{l_1,v}$  that is incident to $l_1$ in $T$ to a path  in $G_n$ which begins with  the edge $e$  in $G_n$ proceeding rightward out 
from $\phi (l_1)$.  
If $T$ has only two nodes altogether, so $n=2$,  we map the node $v$ which is also a leaf to the other endpoint of the edge $e$ in $G_n$, 
and we are done.   For $n>2$,  $v$ cannot be a leaf.   In this case, let  $e_{v,w_1}$ and $e_{v,w_2}$ be the other two edges incident to $v$ besides $e_{v,l_1}$.   Next consider the graph $T - \{ e_{v,w_1}\} $ obtained from $T$ by deleting the edge $e_{v,w_1}$, and let $T_1$ be the connected component of this graph  which contains $w_1$; 
 let $r$ be the number of leaves in the tree   $T_1$.   Let $T_2$ be the tree likewise obtained by deleting the edge $e_{v,w_2}$ from $T$ and taking the connected component containing $w_2$;   one may observe that $n-r-1$ is exactly  the number of leaves in the tree  $T_2$,  since each leaf of $T$ other than $l_1$ belongs to exactly one of these two trees $T_1$ and $T_2$ and since all leaves appearing  in $T_1$ or $T_2$  are leaves of $T$ not equalling $l_1$ and not equalling each other.  We now show how to reduce the  desired spanning tree construction to two strictly  smaller questions of constructing suitable 
 spanning trees within $G_r$ (resp. $G_{n-r-1}$) from which  $T_1$ (resp. $T_2$) may be obtained by a series of edge contractions.  For this 
 construction, it will be important  to choose $T_1$ in such a way that we have $r\le n-r-1$;  this ensures that the 
 subtree $T_1$ which  has (weakly)  
 fewer leaves than $T_2$ is  the subtree that  gets  embedded in $G_n$  by first proceeding one level higher within $G_n$   from $v$ to $v_u$
 in our upcoming construction.    
 One may observe below that this  choice ensures that the embedding will not 
 require more vertical steps upward in $G_n$ than are available within $G_n$.
 
Let $\phi (v)$ be the node in $G_n$ obtained by traversing a path of exactly $r$ edges all proceeding to the right from   $\phi (l_1)$.  
We map the edge 
$e_{l_1,v} \in T$ to the horizontal path with $r$ edges proceeding to the right from $\phi (l_1)$ to $\phi (v)$.    We map the two edges
$e_{v,w_1}$ and $e_{v,w_2}$ outward from $v$ to two  paths outward from $\phi (v)$,  the former of which starts with  the edge $e_{\phi (v),v_u}$ going upward from 
$\phi (v)$ in $G_n$ to the node $v_u$ just above $v$  and the latter of which starts with  the edge $e_{\phi (v),v_r}$ going rightward from $\phi (v)$ in $G_n$ to the node $v_r$ to the immediate right of $v$.   We will use the tree $T_1$ to generate a spanning tree for the portion $G_n^u$ of $G_n$ appearing upward and to the left of $v_u$ (including $v_u$ itself in this spanning tree), while we will use the tree $T_2$ to produce a spanning tree for the portion $G_n^r$  of $G_n$ appearing upward and to the right of $v_r$ (again including $v_r$ itself in this tree).   
To this end, observe that 
$G_{n-r-1} -  \{ v_{leftmost} \} $ sits inside $G_n^r$  as an induced subgraph comprising the bottom-most portion of $G_n^r$.
Observe likewise
that  $G_{n-r-1} -  \{ v_{rightmost  } \}  $ sits inside $G_n^u$ as an induced subgraph by doing a right-left reflection of the graph  
$G_{n-r-1}$  and then using the bottom-most portion of $G_n^u$.  
 For the former, we inductively repeat the construction given so far, now letting $v_r$ play the role originally played by the node to the immediate right of $\phi (l_1)$ in $G_n$, using induction to construct a spanning tree of the desired form for  $G_{n-r-1}$; we then augment this with (possibly empty)  paths straight upward from each top node in $G_{n-r-1} -  \{ v_{leftmost} \} $ to the top node directly above it in $G_n^r$, thereby sending the leaves in $T_2$ to exactly the nodes running along the top of $G_n^r$ and using all nodes of $G_n^r$ so we indeed get a spanning tree  from which $T_2$ may be obtained by a series of edge contractions.  We use exactly the same construction to obtain a spanning tree for $G_n^u$ from which $T_1$ may be obtained by a series of  edge contractions, now flipping the role of right and left so as to proceed leftward and upward within $G_n^u$ to construct its  spanning tree.   Observe by virtue of our construction that $v_u$ is not a top node of $G_n$ unless we have $r=1$, ensuring by induction that the embedding process will never get stuck or need to go upward beyond the top of $G_n$.

Notice that each node in $G_n$ outside of the segment from $\phi (l_1)$ to $\phi (v)$ will belong to exactly of the two graphs $G_n^r$ or $G_n^u$.  In particular, our construction shows that 
the path from $\phi (l_1)$ to $\phi (v)$  together with  the edges from $\phi (v)$ to $v_r$ and from $\phi (v)$ to  
$v_h$ together with the spanning trees for $G_n^r$ and $G_n^u$ will all fit together  give the desired spanning tree for $G_n$. 
\end{proof}

\end{document}